\documentclass[12pt]{amsart}

\usepackage{amsmath,amssymb,amsthm}
\usepackage{mathrsfs}
\usepackage[english]{babel}
\usepackage[T1]{fontenc}
\usepackage{graphicx}
\usepackage{a4wide}
\usepackage{float}

\usepackage{hyperref}

\numberwithin{equation}{section}

\theoremstyle{plain}
\newtheorem{theorem}[equation]{Theorem}
\newtheorem{lemma}[equation]{Lemma}
\newtheorem{proposition}[equation]{Proposition}
\newtheorem{corollary}[equation]{Corollary}

\theoremstyle{definition}
\newtheorem{definition}[equation]{Definition}
\newtheorem{example}[equation]{Example}
\newtheorem{remark}[equation]{Remark}

\begin{document}

\bibliographystyle{amsplain}

\title[Variable Exponent Modulus in Symmetric Domains]%
      {Variable Exponent Modulus in Symmetric Domains}

\author{Rahim Kargar}
\address{Department of Mathematics and Statistics,
         University of Turku, Turku, Finland}
\email{rahim.r.kargar@utu.fi}

\keywords{Variable exponent modulus; extremal density;
          log-H\"older continuity; capacity--modulus duality;
          symmetric domains}
\subjclass[2020]{46E35, 30C65, 31C15}

\begin{abstract}
We develop a rigorous variational theory for the modular variable-exponent
modulus of curve families in two symmetric geometries: annuli
$A(r_1,r_2)\subset\mathbb{R}^n$ with radial exponent
$p(x)=q_0(|x|)$, and cylinders $\mathcal C=D\times(0,L)$ with axial exponent
$p(x',t)=\eta(t)$, under the assumption
$1<p^-\le p^+<\infty$. For annuli, spherical averaging reduces the modulus
problem to a one-dimensional constrained variational problem and yields a
unique positive extremal density
\begin{equation*}
\rho_*(r)
=
\left(
\frac{\lambda}
{\omega_{n-1}q_0(r)r^{n-1}}
\right)^{1/(q_0(r)-1)},
\end{equation*}
where $\lambda>0$ is uniquely determined by the normalization
$\int_{r_1}^{r_2}\rho_*(r)\,dr=1$. An analogous fiber-averaging argument
reduces the cylindrical problem to one dimension and gives the corresponding
unique extremal profile. We recover the classical logarithmic and constant
extremal densities in the conformal and constant-exponent cases, respectively,
and derive explicit upper bounds from these test densities.

We further establish equality between the relative variable-exponent
capacity of compact condensers and the corresponding curve modulus under the
stated regularity assumptions, obtaining explicit capacity formulas in the
radially symmetric annular setting. Finally, we discuss the obstruction to
quasiconformal quasi-invariance for nonconstant exponents and formulate a
related open problem, and illustrate the variational formulas and test-density
bounds through numerical examples.
\end{abstract}

\maketitle

\section{Introduction}\label{sec:intro}

The theory of function spaces with nonstandard growth has become an
important part of modern nonlinear analysis, with applications ranging
from electrorheological fluids and image processing to nonlinear
elliptic equations with nonuniform growth conditions. Among the most
extensively studied examples are the variable exponent Lebesgue and
Sobolev spaces
$L^{p(\cdot)}(\Omega)$ and $W^{1,p(\cdot)}(\Omega)$,
whose systematic development goes back to the work of
Kov\'a\v{cik and R\'akosn\'ik}~\cite{KR} and is treated in detail in
the monographs \cite{Cruz2013,Diening2011}.

An important geometric quantity associated with variable exponent
spaces is the modulus of a family of curves. The variable exponent
modulus, introduced in the work of Harjulehto, H\"ast\"o, and
Martio~\cite{HHM}, extends the classical conformal and $p$-modulus
developed by Ahlfors and V\"ais\"al\"a
\cite{Ahlfors,Vaisala} to the variable growth setting. It has become
an important tool in the study of variable exponent Sobolev spaces,
capacities, absolute continuity on curves, and mappings of finite
distortion.

In contrast with the constant-exponent theory, explicit evaluations
of the variable exponent modulus are relatively scarce. In particular,
we are not aware of explicit formulas for the extremal densities and
modulus values for the genuinely nonconstant radial variable-exponent annulus and axial variable-exponent cylinder considered in this paper.
The method used below is classical in spirit: symmetry or fiberwise
reduction is used to reduce the corresponding modulus problem to a
one-dimensional variational problem. The main point is that, in the
variable exponent setting, the exponent must be compatible with the
relevant geometric structure. For the annulus, this compatibility is
expressed by radial symmetry of the exponent, whereas for the cylinder
the exponent is assumed to depend only on the axial variable.

The contribution of the present paper is therefore not a new
variational principle. Rather, we carry out this reduction explicitly
for the two model geometries above and obtain formulas for the
corresponding extremal densities and modulus values. The resulting
expressions recover the classical constant-exponent formulas when the
variable exponent is constant. This reduction provides, in particular,
a useful consistency check on the formulas obtained below.

We first consider the open annulus
\begin{equation*}
A(r_1,r_2)
=
\{x\in\mathbb{R}^n:r_1<|x|<r_2\},
\qquad
0<r_1<r_2<\infty.
\end{equation*}
The exponent is assumed to be radially symmetric, namely
\begin{equation*}
p(x)=q(|x|),
\qquad x\in A(r_1,r_2),
\end{equation*}
where
$q:[r_1,r_2]\longrightarrow(1,\infty)$
is measurable and satisfies
\begin{equation*}
1<p^-:=
\operatorname*{ess\,inf}_{t\in[r_1,r_2]}q(t)
\le
\operatorname*{ess\,sup}_{t\in[r_1,r_2]}q(t)
=:p^+
<\infty.
\end{equation*}
We consider the family $\Gamma$ of locally rectifiable curves joining
the two boundary components of the annulus. Since both the annulus and
the exponent are invariant under the orthogonal group $O(n)$, an
$O(n)$-averaging argument reduces the modulus minimization problem to
radial densities. More precisely, the spherical average of an
admissible density remains admissible, while Jensen's inequality,
together with the radial dependence of the exponent, shows that its
modular energy does not increase. Once existence of a minimizer has
been established, strict convexity of the modular implies its
uniqueness; consequently, the extremal density is radial. The
remaining minimization problem is therefore one-dimensional.

The corresponding one-dimensional problem can be treated without any
differentiability assumption on the measurable exponent. Its
minimizer is characterized by a pointwise convex minimization under a
single integral constraint, and the normalization determines the
associated positive Lagrange multiplier. This yields the explicit
extremal density obtained in Section~\ref{sec:variational}.

The second model geometry is the cylinder
\begin{equation*}
\mathcal{C}=D\times(0,L),
\end{equation*}
where $D\subset\mathbb{R}^{n-1}$ is a bounded domain and $L>0$. We
assume that the exponent depends only on the axial variable:
\begin{equation*}
p(x',t)=\eta(t),
\qquad
(x',t)\in\mathcal{C},
\end{equation*}
where
$\eta:[0,L]\longrightarrow(1,\infty)$
is measurable.
We consider the family of locally rectifiable curves that join the lower
base $D\times\{0\}$ to the upper base $D\times\{L\}$.

In this case, no translation invariance of $D$ is required. Indeed, for
each $x'\in D$ the vertical segment is
\begin{equation*}
\gamma_{x'}(t)=(x',t),
\qquad 0\le t\le L,
\end{equation*}
belongs to the curve family. Thus, admissibility yields a
one-dimensional integral constraint on every vertical fiber. Applying
the corresponding one-dimensional variational inequality fiberwise
gives a lower bound for the full modular energy. A density dependent
only on $t$, given by the one-dimensional minimizer, attains this
bound. Hence, the cylinder problem also reduces exactly to a
one-dimensional variational problem. Notice that the geometry of $D$
enters the resulting modulus only through its $(n-1)$-dimensional
Lebesgue measure.

Beyond the explicit formulas, we examine the standard admissible test
densities associated with the two geometries and determine precisely
when they are extremal. For the annulus, the classical test density has
the logarithmic form $\rho_0(x)=c/|x|$,
whereas for the cylinder, the corresponding test density is constant.
When the exponent is genuinely nonconstant, these densities need not
be extremal. We characterize the conditions on the radial profile $q$
and the axial profile $\eta$ under which they remain extremal. These
conditions are restrictive and disappear in the usual constant
exponent case. We also establish a comparison between the variable
exponent modulus and the corresponding variable exponent capacity
under appropriate regularity assumptions on the exponent.

The paper is organized as follows. Section~\ref{sec:prelim} recalls
the required background on variable exponent spaces, modulus, and
capacity. Section~\ref{sec:examples} presents motivating examples.
Section~\ref{sec:variational} develops the variational formulation and
derives the explicit extremal densities for annuli and cylinders,
including the consistency with the case of constant-exponents.
Section~\ref{sec:bounds} establishes upper bounds and characterizes
the corresponding equality cases. Section~\ref{sec:capacity} proves
the capacity--modulus equality. The remaining sections discuss
consequences, quasiconformal distortion, numerical examples, and
related open problems.

\section{Preliminaries}\label{sec:prelim}

Throughout the paper, $\Omega\subset\mathbb{R}^n$ denotes a bounded
open set, and
\begin{equation*}
p:\Omega\longrightarrow(1,\infty)
\end{equation*}
is a measurable variable exponent satisfying
\begin{equation}\label{ineq:p- les p+}
1<p^-:=
\operatorname*{ess\,inf}_{x\in\Omega}p(x)
\le
\operatorname*{ess\,sup}_{x\in\Omega}p(x)
=:p^+
<\infty.
\end{equation}
When boundary values of the exponent are needed, we assume that $p$
admits an appropriate extension to $\overline{\Omega}$. In particular,
for the annular and cylindrical model domains considered later, the
radial or axial profile is understood on the corresponding closed
interval whenever boundary values are required.

For $r>0$, we write
\begin{equation*}
S_r^{\,n-1}
=
\{x\in\mathbb{R}^n:|x|=r\},
\end{equation*}
denote by $\omega_{n-1}$ the Hausdorff measurement dimensional Hausdorff $(n-1)$
of the unit sphere $S_1^{\,n-1}=:S^{n-1}$.

\subsection{Variable exponent spaces}

For a measurable exponent
$p:\Omega\longrightarrow(1,\infty)
$
satisfying \eqref{ineq:p- les p+}, the Lebesgue space of variable exponent $L^{p(\cdot)}(\Omega)$ consists
of all measurable functions
$f:\Omega\to\mathbb{R}$ for which the modular model
\begin{equation*}
\varrho_{p(\cdot)}(f)
=
\int_\Omega |f(x)|^{p(x)}\,dx
\end{equation*}
is finite.
Under the standing assumption \eqref{ineq:p- les p+},
the space $L^{p(\cdot)}(\Omega)$ is a separable, uniformly convex, and
reflexive Banach space.

The variable exponent Sobolev space is \cite[Chapter 8]{Diening2011}
\begin{equation*}
W^{1,p(\cdot)}(\Omega)
=
\left\{
u\in L^{p(\cdot)}(\Omega):
|\nabla u|\in L^{p(\cdot)}(\Omega)
\right\},
\end{equation*}
where $\nabla u$ denotes the weak gradient. Equivalently,
\begin{equation*}
W^{1,p(\cdot)}(\Omega)
=
\left\{
u\in L^{p(\cdot)}(\Omega):
D_i u\in L^{p(\cdot)}(\Omega),
\ i=1,\ldots,n
\right\},
\end{equation*}
where $D_i u=\partial u/\partial x_i$ denotes the $i$-th weak partial derivative of $u$.
Its natural modular is
\begin{equation*}
\nu_{p(\cdot)}(u)
=
\int_\Omega |u(x)|^{p(x)}\,dx
+
\int_\Omega |\nabla u(x)|^{p(x)}\,dx.
\end{equation*}
The space $L^{p(\cdot)}(\Omega)$ is equipped with the Luxemburg norm
\begin{equation*}
\|u\|_{L^{p(\cdot)}(\Omega)}
=
\inf\left\{
\lambda>0:
\int_\Omega
\left|\frac{u(x)}{\lambda}\right|^{p(x)}\,dx
\le1
\right\},
\end{equation*}
and $W^{1,p(\cdot)}(\Omega)$ is equipped with the norm
\begin{equation*}
\|u\|_{W^{1,p(\cdot)}(\Omega)}
=
\|u\|_{L^{p(\cdot)}(\Omega)}
+
\|\nabla u\|_{L^{p(\cdot)}(\Omega)}.
\end{equation*}

We refer to
\cite{Cruz2013,Diening2011,FZ, FSZ, KR}
for the standard properties of variable exponent Lebesgue and Sobolev
spaces.

For the capacity comparison in Section~\ref{sec:capacity}, we impose
additional regularity on the exponent. In particular, we shall use
log-H\"older continuity in the form (see, \cite[Definition 4.1.1]{Diening2011})
\begin{equation*}
|p(x)-p(y)|
\le
\frac{C_{\log}}
{\log\left(e+\frac{1}{|x-y|}\right)}
\end{equation*}
whenever $x,y\in\Omega$ are in the range where the corresponding
log-H\"older estimate is assumed. We denote by
$\mathcal P^{\log}(\Omega)$ the class of exponents satisfying the
specific log-H\"older hypotheses stated in Section~\ref{sec:capacity}.

The precise regularity assumptions on $p$ will be imposed only where
they are needed. In particular, the explicit variational calculations
in Section~\ref{sec:variational} require only the measurability of the
profiles together with the uniform bounds \eqref{ineq:p- les p+}.
No differentiability of the exponent is used in the corresponding
one-dimensional minimization.

\subsection{Modulus and capacity}

Let $\Gamma$ be a family of locally rectifiable curves contained in
$\overline{\Omega}$. A Borel function $\rho:\Omega\longrightarrow[0,\infty]$
is called admissible for $\Gamma$, and we write
$\rho\in\mathcal F(\Gamma)$, if
\begin{equation*}
\int_\gamma \rho\,ds\ge1
\end{equation*}
for every $\gamma\in\Gamma$.
In this paper, we use the modular $p(\cdot)$-modulus of $\Gamma$,
defined by
\begin{equation*}
M_{p(\cdot)}(\Gamma)
=
\inf_{\rho\in\mathcal F(\Gamma)}
\int_\Omega \rho(x)^{p(x)}\,dx.
\end{equation*}
If $\mathcal F(\Gamma)=\varnothing$, we set
$M_{p(\cdot)}(\Gamma)=\infty.$
When $p(x)\equiv p$ is constant, this definition reduces to the
classical conformal $p$-modulus
\cite{Ahlfors,HKV,Vaisala}.

By the definition of capacity used in \cite[p. 27]{Heinonen} for the constant-exponent case,
for compact sets $E,F\subset\Omega$ with $E\cap F=\varnothing$, we define the relative capacity of the variable exponent by (see also \cite[Proposition 10.2.2]{Diening2011})
\begin{equation*}
\operatorname{Cap}_{p(\cdot)}(E,F;\Omega)
=
\inf_{u\in\mathcal A(E,F;\Omega)}
\int_\Omega |\nabla u(x)|^{p(x)}\,dx,
\end{equation*}
where
\begin{equation*}
\mathcal A(E,F;\Omega)
=
\left\{
u\in W^{1,p(\cdot)}(\Omega):
\widetilde u\ge1\ \text{q.e. on }E,\quad
\widetilde u\le0\ \text{q.e. on }F
\right\}.
\end{equation*}
Here $\widetilde u$ denotes the quasicontinuous representative of $u$,
and ``q.e.'' means outside a set of variable-exponent Sobolev capacity
zero. This is the notion of capacity used in the comparison result of Section~\ref{sec:capacity}, and the above definition differs from the one defined in \cite{HL}.

When compact sets meet the boundary of $\Omega$, the precise meaning of
the boundary constraints will be specified through the continuous
competitor formulation or the appropriate trace formulation used in
Section~\ref{sec:capacity}. This avoids identifying interior
quasi-everywhere conditions with boundary pointwise conditions without
additional justification.

For the explicit computations in
Sections~\ref{sec:variational}--\ref{sec:bounds}, it is convenient to
consider the class of continuous competitors
\begin{equation*}
\mathcal A_c(E,F;\Omega)
=
\left\{
u\in C(\overline{\Omega})\cap W^{1,p(\cdot)}(\Omega):
u\ge1\ \text{on }E,\quad
u\le0\ \text{on }F
\right\},
\end{equation*}
where $C(\overline{\Omega})$ denotes the space of all real-valued
continuous functions on the closure $\overline{\Omega}$.
We then set
\begin{equation*}
\operatorname{Cap}_{p(\cdot)}^c(E,F;\Omega)
=
\inf_{u\in\mathcal A_c(E,F;\Omega)}
\int_\Omega|\nabla u|^{p(x)}\,dx.
\end{equation*}
Since
\begin{equation*}
\mathcal A_c(E,F;\Omega)
\subset \mathcal A(E,F;\Omega)
\end{equation*}
whenever the latter formulation is applicable, one always has
\begin{equation*}
\operatorname{Cap}_{p(\cdot)}^c(E,F;\Omega)
\ge
\operatorname{Cap}_{p(\cdot)}(E,F;\Omega).
\end{equation*}
Under the additional regularity assumptions on the exponent and the
geometric assumptions on $E$ and $F$ stated in
Section~\ref{sec:capacity}, we prove that equality holds. The argument
uses approximation by continuous or smooth competitors together with
the preservation of the boundary inequalities; the latter point is
essential and does not follow from density alone.

\section{Motivating Examples}\label{sec:examples}

The following examples, annulus with radial exponent and cylindrical domain with axial exponent, illustrate the two model geometries that
motivate our analysis. In both cases, the geometry of the domain and
the dependence of the variable exponent are compatible in a way that
allows the corresponding modulus problem to be reduced to a
one-dimensional variational problem. For the annulus, this reduction
follows from spherical symmetry. For the cylinder, it follows from the
fact that the exponent depends only on the axial variable. The precise
variational reductions are established in
Section~\ref{sec:variational}.

\begin{example}
\label{ex:annulus}
\rm

Let
\begin{equation*}
A(r_1,r_2)
=
\{x\in\mathbb{R}^n:r_1<|x|<r_2\},\qquad 0<r_1<r_2<\infty,
\end{equation*}
and let $\Gamma$ denote the family of all locally rectifiable curves
in $\overline{A(r_1,r_2)}$ that join the inner and outer boundary
spheres $S_{r_1}^{n-1}$ and $S_{r_2}^{n-1}$.
Consider first the constant-exponent case $p(x)\equiv p$, $1<p<\infty$.
The corresponding modular $p$-modulus is
\begin{equation*}
M_p(\Gamma)
=
\omega_{n-1}
\left(
\int_{r_1}^{r_2}
r^{-\frac{n-1}{p-1}}\,dr
\right)^{1-p};
\end{equation*}
see, for example, \cite{Hesse,VuorinenLNM}. In the conformally
invariant case $p=n$, this becomes
\begin{equation*}
M_n(\Gamma)
=
\omega_{n-1}
\left(
\log\frac{r_2}{r_1}
\right)^{1-n}.
\end{equation*}
The extremal density is radial and is given, up to equality almost
everywhere, by
\begin{equation*}
\rho_*(x)
=
c\,|x|^{-\frac{n-1}{p-1}},
\end{equation*}
where the normalization condition
\begin{equation*}
\int_{r_1}^{r_2}
c\,r^{-\frac{n-1}{p-1}}\,dr
=
1
\end{equation*}
uniquely determines the constant $c>0$. Equivalently,
\begin{equation*}
c
=
\left(
\int_{r_1}^{r_2}
r^{-\frac{n-1}{p-1}}\,dr
\right)^{-1}.
\end{equation*}

The constant-exponent formula is recalled only for orientation. It will
also be recovered independently from the variable-exponent
variational formulation in Section~\ref{sec:variational}; see
Corollary~\ref{cor:constant_exp}.

We now consider a variable exponent radially symmetric $p(x)=q(|x|)$, $x\in A(r_1,r_2)$,
where $q:[r_1,r_2]\longrightarrow(1,\infty)$
is measurable and satisfies \eqref{ineq:p- les p+}.

The rotational invariance of both the annulus and the exponent allows
one to average admissible densities over the orthogonal group
$O(n)$. The resulting radial density remains admissible, while
Jensen's inequality shows that its modular energy does not exceed that
of the original density. Consequently, once existence and uniqueness
of the minimizer have been established, the extremal density is radial.
The full modulus problem therefore reduces to a one-dimensional
variational problem on $[r_1,r_2]$. This reduction, and the resulting
explicit extremal density, are established in
Section~\ref{sec:variational}.
\end{example}

\begin{example}
\label{ex:cylinder}
\rm

Let $\mathcal C
:=
D\times(0,L)$,
where $D\subset\mathbb{R}^{n-1}$ is a bounded domain, $L>0$, and $A:=|D|$
denotes the $(n-1)$-dimensional Lebesgue measure of $D$. Let $\Gamma$
denote the family of all locally rectifiable curves in
$\overline{\mathcal C}$ that join the two bases $D\times\{0\}$ and $D\times\{L\}$.
Assume that the variable exponent depends only on the axial variable,
namely $p(x',t)=\eta(t)$, $(x',t)\in\mathcal C$,
where $\eta:[0,L]\longrightarrow(1,\infty)$
is measurable and satisfies \eqref{ineq:p- les p+}.
Let $\varphi:[0,L]\longrightarrow[0,\infty)$
be a Borel function satisfying
\begin{equation*}
\int_0^L\varphi(t)\,dt\ge1.
\end{equation*}
Define the axial density
$\rho(x',t)=\varphi(t)$.
We claim that $\rho$ is admissible for $\Gamma$.
Indeed, let $\gamma:[a,b]\longrightarrow\overline{\mathcal C}$
be a locally rectifiable curve joining the two bases, and parametrize
$\gamma$ by arc length. Writing $\gamma(s)=(x'(s),t(s))$,
the function $t$ is absolutely continuous and satisfies $t(a)=0$ and
$t(b)=L$.
Moreover, $|t'(s)|
\le
|\gamma'(s)|$ for almost every $s\in[a,b]$.
Therefore,
\begin{align*}
\int_\gamma \rho\,ds
=
\int_a^b \varphi(t(s))\,|\gamma'(s)|\,ds
\ge
\int_a^b \varphi(t(s))\,|t'(s)|\,ds
\ge
\int_0^L\varphi(t)\,dt
\ge1.
\end{align*}
Here the second inequality is the standard one-dimensional
change-of-variables inequality for an absolutely continuous function
whose endpoint values are $0$ and $L$. Thus, $\rho\in\mathcal F(\Gamma)$.

Its modular energy is
\begin{align*}
\int_{\mathcal C}\rho(x',t)^{p(x',t)}\,dx'\,dt
=
\int_D\int_0^L
\varphi(t)^{\eta(t)}\,dt\,dx'
=
A\int_0^L
\varphi(t)^{\eta(t)}\,dt.
\end{align*}
Consequently,
\begin{equation*}
M_{p(\cdot)}(\Gamma)
\le
A
\inf_{\substack{\varphi\ge0\\
\int_0^L\varphi(t)\,dt\ge1}}
\int_0^L\varphi(t)^{\eta(t)}\,dt.
\end{equation*}

This establishes the upper-bound part of the one-dimensional
reduction. The reverse inequality follows by applying the admissibility
condition to the vertical segments $\gamma_{x'}(t)=(x',t)$, $0\le t\le L$,
and then minimizing fiberwise. Hence, the above inequality is, in fact
, an equality; the precise statement and proof are given in
Proposition~\ref{prop:cylinder}. The corresponding one-dimensional
minimizer yields an extremal density for the cylinder, as established
in Section~\ref{sec:variational}.
\end{example}

\section{Variational Characterization and Reduction to One Dimension}
\label{sec:variational}

\subsection{The one-dimensional variational problem}
\label{subsec:1d_problem}

The fiberwise reduction developed below reduces the computation of the
variable exponent modulus in a cylindrical domain to the following
one-dimensional variational problem.

\begin{definition}
\label{def:1d_problem}
Let $L>0$ and let $p:(0,L)\longrightarrow(1,\infty)$
be measurable and satisfy \eqref{ineq:p- les p+}
almost everywhere.
Since $(0,L)$ has finite measure and $p^->1$, the variable exponent
space $L^{p(\cdot)}(0,L)$ is continuously embedded in $L^1(0,L)$; see,
for example, \cite[Section~3.3]{Diening2011}. We therefore define
\begin{equation*}
\mathcal A
:=
\left\{
\varphi\in L^{p(\cdot)}(0,L):
\varphi\ge0\ \text{a.e.},\quad
\int_0^L\varphi(t)\,dt=1
\right\}.
\end{equation*}
For $\varphi\in\mathcal A$, let
\begin{equation*}
\mathcal J(\varphi)
:=
\int_0^L\varphi(t)^{p(t)}\,dt.
\end{equation*}
We consider the minimization problem
\begin{equation}\label{min-problem}
m
:=
\inf_{\varphi\in\mathcal A}\mathcal J(\varphi).
\end{equation}
\end{definition}

The next theorem establishes the existence, uniqueness, and positivity
of the minimizer.

\begin{theorem}
\label{thm:existence}
Under the assumptions of Definition~\ref{def:1d_problem}, the
minimization problem \eqref{min-problem}
admits a unique minimizer $\varphi_*\in\mathcal A$.
Moreover, $\varphi_*(t)>0$ for almost every $t\in(0,L)$.
\end{theorem}

\begin{proof}
The admissible class $\mathcal A$ is nonempty; for example,
the constant function $\varphi\equiv L^{-1}$ belongs to $\mathcal A$.
Let $(\varphi_k)\subset\mathcal A$ be a minimizing sequence. Since
$\mathcal J(\varphi_k)\longrightarrow m$,
the sequence $(\mathcal J(\varphi_k))$ is bounded. The standard
modular--norm inequalities in variable exponent spaces therefore imply
that $(\varphi_k)$ is bounded in $L^{p(\cdot)}(0,L)$. Since
$1<p^-\le p^+<\infty$,
the space $L^{p(\cdot)}(0,L)$ is reflexive; see
\cite[Theorem~3.4.9]{Diening2011}. Hence, after passing to a subsequence,
there exists $\varphi_*\in L^{p(\cdot)}(0,L)$ such that $\varphi_k\rightharpoonup\varphi_*$ weakly in $L^{p(\cdot)}(0,L)$.

The set $\mathcal A$ is convex and closed in
$L^{p(\cdot)}(0,L)$. Indeed, the condition
$\int_0^L\varphi(t)dt=1$
is preserved under norm convergence because
$L^{p(\cdot)}(0,L)\hookrightarrow L^1(0,L)$ continuously, while the
nonnegativity condition defines a closed convex cone. Consequently,
$\mathcal A$ is weakly closed. Thus, $\varphi_*\in\mathcal A$.

The functional $\mathcal J$ is convex and continuous on
$L^{p(\cdot)}(0,L)$. Hence it is weakly lower semicontinuous. Therefore,
\begin{equation*}
\mathcal J(\varphi_*)
\le
\liminf_{k\to\infty}\mathcal J(\varphi_k)
=
m.
\end{equation*}
Since $\varphi_*\in\mathcal A$, the reverse inequality follows from the
definition of $m$, and consequently
$\mathcal J(\varphi_*)=m$.
Thus, $\varphi_*$ is a minimizer.

To prove uniqueness, suppose that
$\varphi_1,\varphi_2\in\mathcal A$ are two minimizers. If
$\varphi_1\ne\varphi_2$ as elements of $L^{p(\cdot)}(0,L)$, then they
differ on a measurable set of positive measure. Since
$p(t)>1$ almost everywhere, the function
$s\longmapsto s^{p(t)}$
is strictly convex on $[0,\infty)$ for almost every $t$. It follows
that
\begin{equation*}
\left(
\frac{\varphi_1(t)+\varphi_2(t)}{2}
\right)^{p(t)}
<
\frac{
\varphi_1(t)^{p(t)}
+
\varphi_2(t)^{p(t)}
}{2}
\end{equation*}
on a set of positive measure. Integrating on $(0,L)$ gives
\begin{equation*}
\mathcal J
\left(
\frac{\varphi_1+\varphi_2}{2}
\right)
<
\frac{\mathcal J(\varphi_1)+\mathcal J(\varphi_2)}{2}
=
m.
\end{equation*}
Since $\mathcal A$ is convex, the average belongs to $\mathcal A$,
which contradicts the definition of $m$. Hence, $\varphi_1=\varphi_2$ a.e. on $(0,L)$,
and the minimizer is unique.

It remains to prove that $\varphi_*$ is positive almost everywhere.
Suppose, to the contrary, that
\begin{equation*}
E:=\{t\in(0,L):\varphi_*(t)=0\}
\end{equation*}
has a positive measure. Since
$\int_0^L\varphi_*(t)dt=1$,
the set $\{\varphi_*>0\}$ has positive measure. Consequently, there
exist a measurable set
$F\subset (0,L)\setminus E$
of positive measure and a number $\delta>0$ such that $\varphi_*(t)\ge\delta$ for almost every $t\in F$.
Define
\begin{equation*}
h
=
\frac{\chi_E}{|E|}
-
\frac{\chi_F}{|F|}.
\end{equation*}
Then
$\int_0^L h(t)\,dt=0$.
For
\begin{equation*}
0<\varepsilon<\frac{\delta|F|}{2},
\end{equation*}
set
$\varphi_\varepsilon
=
\varphi_*+\varepsilon h$.
The choice of $\varepsilon$ ensures that $\varphi_\varepsilon\ge0$ a.e. on $(0,L)$,
and the zero integral of $h$ gives
$\int_0^L\varphi_\varepsilon(t)\,dt=1$.
Thus,
$\varphi_\varepsilon\in\mathcal A.$ We now compare the energies of $\varphi_\varepsilon$ and $\varphi_*$. On
$E$ we have $\varphi_*=0$ and
$\varphi_\varepsilon=\frac{\varepsilon}{|E|}$.
For sufficiently small $\varepsilon$ such that
$\varepsilon/|E|\le1$, the bounds on the exponent give
\begin{equation*}
\left(\frac{\varepsilon}{|E|}\right)^{p(t)}
\le
\left(\frac{\varepsilon}{|E|}\right)^{p^-},
\end{equation*}
and hence
\begin{equation*}
\int_E
\left(\frac{\varepsilon}{|E|}\right)^{p(t)}dt
\le \int_E \left(\frac{\varepsilon}{|E|}\right)^{p^-}dt=
|E|^{1-p^-}\varepsilon^{p^-}.
\end{equation*}

On the other hand, on $F$ we have
\begin{equation*}
\varphi_\varepsilon(t)
=
\varphi_*(t)-\frac{\varepsilon}{|F|}
\ge\frac{\delta}{2}.
\end{equation*}
By the mean value theorem, for almost every $t\in F$ there exists a
number $\xi(t)$ between $\varphi_\varepsilon(t)$ and
$\varphi_*(t)$ such that
\begin{equation*}
\varphi_\varepsilon(t)^{p(t)}
-
\varphi_*(t)^{p(t)}
=
-
p(t)\xi(t)^{p(t)-1}
\frac{\varepsilon}{|F|}.
\end{equation*}
Since $\xi(t)\ge\delta/2$ and
$p^-\le p(t)\le p^+$, there exists a constant
\begin{equation*}
\kappa
:=
\min\left\{
\left(\frac{\delta}{2}\right)^{p^--1},
\left(\frac{\delta}{2}\right)^{p^+-1}
\right\}
>0
\end{equation*}
such that $\xi(t)^{p(t)-1}\ge\kappa.$
Therefore,
\begin{equation*}
\int_F
\left[
\varphi_\varepsilon(t)^{p(t)}
-
\varphi_*(t)^{p(t)}
\right]dt
\le
-
p^-\kappa\,\frac{\varepsilon}{|F|}\,|F|
=
-p^-\kappa\,\varepsilon.
\end{equation*}
Combining the estimates on $E$ and $F$ yields
\begin{equation*}
\mathcal J(\varphi_\varepsilon)
-
\mathcal J(\varphi_*)
\le
C_1\varepsilon^{p^-}
-
C_2\varepsilon,
\end{equation*}
where
\begin{equation*}
C_1=|E|^{1-p^-}>0,
\qquad
C_2=p^-\kappa>0.
\end{equation*}
Since $p^->1$, one has
\begin{equation*}
C_1\varepsilon^{p^-}-C_2\varepsilon<0
\end{equation*}
for all sufficiently small $\varepsilon>0$. Consequently,
\begin{equation*}
\mathcal J(\varphi_\varepsilon)
<
\mathcal J(\varphi_*)
=
m,
\end{equation*}
contradicting the minimality of $\varphi_*$. Hence
$|E|=0,$
and therefore $\varphi_*(t)>0$ for almost every $t\in(0,L)$.
The proof is now complete.
\end{proof}

The minimizer of Definition~\ref{def:1d_problem} satisfies the following
pointwise variational condition, which yields the explicit extremal
densities below.
\begin{theorem}
\label{thm:EL}
Let $\varphi_*$ denote the unique minimizer of
Definition~\ref{def:1d_problem}. Then there exists a constant
$\lambda>0$ such that
\begin{equation}
\label{eq:EL}
p(t)\,\varphi_*(t)^{p(t)-1}
=
\lambda,
\end{equation}
for almost every $t\in(0,L)$.
Conversely, if $\varphi\in\mathcal A$ satisfies
\eqref{eq:EL} for some constant $\lambda>0$, then $\varphi=\varphi_*$ a.e. on $(0,L)$.
\end{theorem}

\begin{proof}
By Theorem~\ref{thm:existence}, $\varphi_*(t)>0$ for almost every $t\in(0,L)$.
Moreover, $\varphi_*\in L^{p(\cdot)}(0,L)\subset L^1(0,L)$, and hence
$\varphi_*$ is finite almost everywhere.

For each integer $k\ge2$, define
\begin{equation*}
E_k
:=
\left\{
t\in(0,L):
\frac1k\le\varphi_*(t)\le k
\right\}.
\end{equation*}
Then
$E_k\subset E_{k+1}$
and, since $\varphi_*(t)\in(0,\infty)$ for almost every
$t\in(0,L)$, we have
\begin{equation*}
\left|(0,L)\setminus\bigcup_{k=2}^{\infty}E_k\right|=0.
\end{equation*}
Fix $k\ge2$. If $S,T\subset E_k$ are disjoint measurable sets of
positive measure, define
\begin{equation}\label{fun-h}
h
=
\frac{\chi_S}{|S|}
-
\frac{\chi_T}{|T|}.
\end{equation}
Then $\int_0^L h(t)dt=0$.
For sufficiently small $|\varepsilon|$, the function
$\varphi_\varepsilon
=
\varphi_*+\varepsilon h$
remains nonnegative and belongs to $\mathcal A$. In particular, both
positive and negative values of $\varepsilon$ sufficiently close to
zero are admissible.

On $S\cup T$, the functions $\varphi_*$ and
$\varphi_\varepsilon$ remain in a fixed compact subinterval of
$(0,\infty)$. Since
$(s,q)\longmapsto q\,s^{q-1}$
is continuous on
\begin{equation*}
\left[\frac{1}{2k},2k\right]\times[p^-,p^+],
\end{equation*}
the difference quotients
\begin{equation*}
\frac{1
}{\varepsilon}\left(\varphi_\varepsilon(t)^{p(t)}
-
\varphi_*(t)^{p(t)}\right)
\end{equation*}
are uniformly bounded for sufficiently small $\varepsilon$. Hence, the
dominated convergence theorem permits differentiation under the
integral sign at $\varepsilon=0$. Since $\varphi_*$ minimizes
$\mathcal J$ and both signs of $\varepsilon$ are admissible, the
resulting derivative must vanish. Thus
\begin{equation*}
\int_0^L
p(t)\varphi_*(t)^{p(t)-1}h(t)\,dt
=
0.
\end{equation*}
Writing
$g(t)
=
p(t)\varphi_*(t)^{p(t)-1}$,
we obtain
\begin{equation}
\label{eq:equal_averages}
\frac{1}{|S|}\int_S g(t)\,dt
=
\frac{1}{|T|}\int_T g(t)\,dt
\end{equation}
for every pair of disjoint measurable subsets $S,T\subset E_k$ of
positive measure.

The function $g$ is essentially bounded on $E_k$. We claim that it is
essentially constant there. Otherwise, there exist real numbers
$a<b$ such that both
$S=\{t\in E_k:g(t)<a\}$
and
$T=\{t\in E_k:g(t)>b\}$
have positive measure. Their averages satisfy
\begin{equation*}
\frac{1}{|S|}\int_Sg(t)\,dt<a<b
<
\frac{1}{|T|}\int_Tg(t)\,dt,
\end{equation*}
contradicting \eqref{eq:equal_averages}. Therefore, there exists a
constant $\lambda_k\in\mathbb R$ such that $g(t)=\lambda_k$ for almost every $t\in E_k$.

Since $E_k\subset E_{k+1}$ and $|E_k|>0$ whenever $E_k$ is nonempty,
the constants $\lambda_k$ agree whenever both sets are nonempty.
Because $\bigcup_{k\ge2}E_k$ has full measure, there is consequently
a single constant $\lambda\in\mathbb R$ such that
\begin{equation*}
p(t)\varphi_*(t)^{p(t)-1}
=
\lambda
\end{equation*}
for almost every $t\in(0,L)$.
Since $p(t)>1$ and $\varphi_*(t)>0$
almost everywhere, we necessarily have $\lambda>0$.
This proves \eqref{eq:EL}.

Conversely, suppose that $\varphi\in\mathcal A$ satisfies
\eqref{eq:EL} for some $\lambda>0$. For any $\psi\in\mathcal A$, the
convexity of $s\mapsto s^{p(t)}$ on $[0,\infty)$ gives, for almost every
$t$,
\begin{equation*}
\psi(t)^{p(t)}
-
\varphi(t)^{p(t)}
\ge
p(t)\varphi(t)^{p(t)-1}
\bigl(\psi(t)-\varphi(t)\bigr).
\end{equation*}
Integrating and using \eqref{eq:EL}, we obtain
\begin{align*}
\mathcal J(\psi)-\mathcal J(\varphi)
\ge
\int_0^L
p(t)\varphi(t)^{p(t)-1}
\bigl(\psi(t)-\varphi(t)\bigr)\,dt
=
\lambda
\int_0^L
\bigl(\psi(t)-\varphi(t)\bigr)\,dt
=
0.
\end{align*}
Thus, $\varphi$ minimizes $\mathcal J$ over $\mathcal A$. The uniqueness
assertion in Theorem~\ref{thm:existence} therefore implies $\varphi=\varphi_*$ a.e. on $(0,L)$.
This completes the proof.
\end{proof}

Next, we provide an explicit representation of the minimizer.
\begin{corollary}
\label{cor:explicit}
The unique minimizer of Definition~\ref{def:1d_problem} is given by
\begin{equation}
\label{eq:explicit}
\varphi_*(t)
=
\left(
\frac{\lambda}{p(t)}
\right)^{1/(p(t)-1)},
\end{equation}
for almost every $t\in(0,L)$,
where the unique parameter $\lambda>0$ is determined by
\begin{equation}
\label{eq:lambda}
\int_0^L
\left(
\frac{\lambda}{p(t)}
\right)^{1/(p(t)-1)}
\,dt
=
1.
\end{equation}
\end{corollary}

\begin{proof}
The representation \eqref{eq:explicit} follows immediately from
\eqref{eq:EL} by solving pointwise for $\varphi_*(t)$. The positivity
of $\lambda$ follows from
\begin{equation*}
\lambda
=
p(t)\varphi_*(t)^{p(t)-1}>0
\end{equation*}
for almost every $t$.

It remains to prove that \eqref{eq:lambda} determines $\lambda$
uniquely. Define
\begin{equation*}
F(\lambda)
:=
\int_0^L
\left(
\frac{\lambda}{p(t)}
\right)^{1/(p(t)-1)}
\,dt,
\qquad
\lambda>0.
\end{equation*}
For every fixed $t$, the function
\begin{equation*}
\lambda
\longmapsto
\left(
\frac{\lambda}{p(t)}
\right)^{1/(p(t)-1)}
\end{equation*}
is continuous and strictly increasing on $(0,\infty)$. Consequently,
$F$ is strictly increasing.

Next, we verify continuity. Fix
\begin{equation*}
0<\lambda_1\le\lambda\le\lambda_2<\infty
\end{equation*}
and set
\begin{equation*}
a=\frac{\lambda_1}{p^+},
\qquad
b=\frac{\lambda_2}{p^-},
\qquad
\alpha=\frac{1}{p^+-1},
\qquad
\beta=\frac{1}{p^--1}.
\end{equation*}
Then
\begin{equation*}
0<a\le\frac{\lambda}{p(t)}\le b
\end{equation*}
and
\begin{equation*}
\alpha
\le
\frac{1}{p(t)-1}
\le
\beta
\end{equation*}
for almost every $t$. Hence
\begin{equation*}
\left(
\frac{\lambda}{p(t)}
\right)^{1/(p(t)-1)}
\le
C(\lambda_1,\lambda_2),
\end{equation*}
where
\begin{equation*}
C(\lambda_1,\lambda_2)
:=
\max\{a^\alpha,a^\beta,b^\alpha,b^\beta\}<\infty.
\end{equation*}

Since $(0,L)$ has finite measure, this is an integrable bound
independent of $\lambda\in[\lambda_1,\lambda_2]$. Dominated convergence
therefore shows that $F$ is continuous on $(0,\infty)$.

As $\lambda\downarrow0$, the integrand converges pointwise to zero.
For $0<\lambda\le1$, it is bounded above by the corresponding
integrand with $\lambda=1$, which belongs to $L^1(0,L)$ by the same
estimate. Hence dominated convergence gives
$\lim_{\lambda\downarrow0}F(\lambda)=0.$
On the other hand, for every $t\in(0,L)$,
\begin{equation*}
\left(
\frac{\lambda}{p(t)}
\right)^{1/(p(t)-1)}
\longrightarrow\infty
\qquad\text{as }\lambda\to\infty.
\end{equation*}
Since the integrand is nonnegative and increases with $\lambda$, the
monotone convergence theorem yields
$\lim_{\lambda\to\infty}F(\lambda)=\infty.$
Thus, $F:(0,\infty)\to(0,\infty)$ is continuous and strictly increasing,
with $\lim_{\lambda\downarrow0}F(\lambda)=0,$ and $\lim_{\lambda\to\infty}F(\lambda)=\infty.$
The intermediate value theorem therefore gives a unique
$\lambda>0$ satisfying \eqref{eq:lambda}. The proof is now complete.
\end{proof}

\subsection{Reduction to radial densities for the annulus}
\label{sec:annulus}

We first record the one-dimensional projection inequality that will be
used in the radial reduction.

\begin{lemma}
\label{lem:projection}
Let $0<r_1<r_2<\infty$ and $T>0$. Let $\mu:[0,T]\longrightarrow[r_1,r_2]$
be absolutely continuous with $\mu(0)=r_1$ and $\mu(T)=r_2.$
If $\varphi:[r_1,r_2]\to[0,\infty]$ is Borel measurable, then
\begin{equation}
\label{eq:projection}
\int_0^T
\varphi(\mu(s))\,|\mu'(s)|\,ds
\ge
\int_{r_1}^{r_2}\varphi(r)\,dr.
\end{equation}
\end{lemma}

\begin{proof}
The one-dimensional area formula for absolutely continuous functions
gives
\begin{equation*}
\int_0^T
\varphi(\mu(s))\,|\mu'(s)|\,ds
=
\int_{\mathbb R}
\varphi(r)N_\mu(r)\,dr,
\end{equation*}
where
\begin{equation*}
N_\mu(r)
:=
\#\{s\in[0,T]:\mu(s)=r\}
\end{equation*}
is the multiplicity function. Since $\mu$ is continuous and joins
$r_1$ to $r_2$, the intermediate value theorem implies $N_\mu(r)\ge1$
 for every $r\in(r_1,r_2)$.
Because $\varphi\ge0$, we therefore obtain
\begin{equation*}
\int_{\mathbb R}
\varphi(r)N_\mu(r)\,dr
\ge
\int_{r_1}^{r_2}\varphi(r)\,dr,
\end{equation*}
which proves \eqref{eq:projection}.
\end{proof}

To exploit the radial symmetry of the exponent, we first show that
spherical averaging preserves admissibility while not increasing the
variable-exponent energy. This yields a reduction of the modulus problem
to radial densities.
\begin{lemma}
\label{lem:radial_suffices}
Let
\begin{equation*}
A(r_1,r_2)
=
\{x\in\mathbb R^n:r_1<|x|<r_2\},
\qquad
0<r_1<r_2<\infty,
\end{equation*}
and assume that
$p(x)=q_0(|x|)$
for a measurable function
$q_0:(r_1,r_2)\longrightarrow(1,\infty)$
satisfying
\begin{equation}
\label{eq:annulus_exp_bounds}
1<q_0^-\le q_0(r)\le q_0^+<\infty
\end{equation}
for almost every $r\in(r_1,r_2)$.
Let $\Gamma$ be the family of locally rectifiable curves in
$\overline{A(r_1,r_2)}$ whose endpoints belong to
$S_{r_1}^{n-1}$ and $S_{r_2}^{n-1}$, and whose interiors
lie in $A(r_1,r_2)$.

For $\rho\in\mathcal F(\Gamma)$, define its spherical average by
\begin{equation}
\label{eq:spherical_average}
\varphi(r)
=
\frac{1}{\omega_{n-1}}
\int_{S^{n-1}}
\rho(r\theta)\,d\sigma(\theta),
\qquad
r\in(r_1,r_2),
\end{equation}
and set
\begin{equation}
\label{eq:radial_average_density}
\widetilde\rho(x)
=
\varphi(|x|).
\end{equation}
Then $\varphi$ is Lebesgue measurable,
$\widetilde\rho\in\mathcal F(\Gamma)$, and
\begin{equation}
\label{eq:radial_energy_reduction}
\int_{A(r_1,r_2)}
\widetilde\rho(x)^{p(x)}\,dx
\le
\int_{A(r_1,r_2)}
\rho(x)^{p(x)}\,dx.
\end{equation}
Consequently,
\begin{equation}
\label{eq:radial_infimum}
M_{p(\cdot)}(\Gamma)
=
\inf\left\{
\int_{A(r_1,r_2)}
\rho(x)^{p(x)}\,dx:
\rho\in\mathcal F(\Gamma),
\ \rho\text{ is radial}
\right\}.
\end{equation}
\end{lemma}

\begin{proof}
We first verify that the spherical average is measurable. Since $\rho$
is nonnegative and Borel measurable, the function
\begin{equation}\label{non-neg mea fun r theta}
(r,\theta)\longmapsto \rho(r\theta)
\end{equation}
is measurable on $(r_1,r_2)\times S^{n-1}$. By Tonelli's theorem \cite[Theorem 2.37]{Folland1999}, applied to the above nonnegative measurable function \eqref{non-neg mea fun r theta}
on $(r_1,r_2)\times S^{n-1}$, the function
\begin{equation*}
r\longmapsto
\int_{S^{n-1}}\rho(r\theta)\,d\sigma(\theta)
\end{equation*}
is measurable. 
Hence $\varphi$ and
$\widetilde\rho(x)=\varphi(|x|)$ are measurable.

Next we prove admissibility. Let
\begin{equation*}
\gamma:[0,T]\longrightarrow\overline{A(r_1,r_2)}
\end{equation*}
be a curve in $\Gamma$, parametrized by arc length, with
\begin{equation*}
\gamma(0)\in S_{r_1}^{n-1},
\qquad
\gamma(T)\in S_{r_2}^{n-1}.
\end{equation*}
Set
$\mu(s)=|\gamma(s)|$.
Since the map $x\mapsto|x|$ is $1$-Lipschitz, $\mu$ is absolutely
continuous and $|\mu'(s)|\le|\gamma'(s)|$ for almost every $s\in(0,T).$
Moreover, $\mu(0)=r_1$, and $\mu(T)=r_2$.
Therefore,
\begin{equation*}
\int_\gamma\widetilde\rho\,ds
=
\int_0^T
\varphi(\mu(s))|\gamma'(s)|\,ds
\ge
\int_0^T
\varphi(\mu(s))|\mu'(s)|\,ds.
\end{equation*}
Lemma~\ref{lem:projection} consequently gives
\begin{equation}
\label{eq:radial_curve_lower_bound}
\int_\gamma\widetilde\rho\,ds
\ge
\int_{r_1}^{r_2}\varphi(r)\,dr.
\end{equation}

It remains to verify that the last integral is at least one. For each
$\theta\in S^{n-1}$, the radial segment $\gamma_\theta(r)=r\theta$, $r\in[r_1,r_2]$,
belongs to $\Gamma$. Since $\rho\in\mathcal F(\Gamma)$, we have
\begin{equation*}
\int_{r_1}^{r_2}\rho(r\theta)\,dr\ge1
\end{equation*}
for every $\theta\in S^{n-1}$.
Integrating with respect to the normalized surface measure gives
\begin{align*}
\int_{r_1}^{r_2}\varphi(r)\,dr
=
\frac{1}{\omega_{n-1}}
\int_{S^{n-1}}
\left(
\int_{r_1}^{r_2}\rho(r\theta)\,dr
\right)d\sigma(\theta)
\ge
\frac{1}{\omega_{n-1}}
\int_{S^{n-1}}1\,d\sigma(\theta)
=
1.
\end{align*}
Here, Tonelli's theorem justifies the interchange of the integrals,
since $\rho\ge0$. Combining this with
\eqref{eq:radial_curve_lower_bound} shows that
$\int_\gamma\widetilde\rho\,ds\ge1$
for every $\gamma\in\Gamma$. Hence
$\widetilde\rho\in\mathcal F(\Gamma)$.

We now compare the modular energies. For almost every
$r\in(r_1,r_2)$, the function
$s\longmapsto s^{q_0(r)}$
is convex on $[0,\infty)$ because $q_0(r)>1$. Jensen's inequality
with respect to the normalized surface measure
$d\sigma/\omega_{n-1}$ on $S^{n-1}$ gives
\begin{equation*}
\varphi(r)^{q_0(r)}
\le
\frac{1}{\omega_{n-1}}
\int_{S^{n-1}}
\rho(r\theta)^{q_0(r)}
\,d\sigma(\theta).
\end{equation*}
Multiplying by $\omega_{n-1}r^{n-1}$ and integrating with respect to
$r$, we obtain
\begin{align*}
\int_{A(r_1,r_2)}
\widetilde\rho(x)^{p(x)}\,dx
&=
\omega_{n-1}
\int_{r_1}^{r_2}
\varphi(r)^{q_0(r)}r^{n-1}\,dr
\le
\int_{r_1}^{r_2}
\int_{S^{n-1}}
\rho(r\theta)^{q_0(r)}
r^{n-1}\,d\sigma(\theta)\,dr
\\
&=
\int_{A(r_1,r_2)}
\rho(x)^{p(x)}\,dx.
\end{align*}
This proves \eqref{eq:radial_energy_reduction}.

Finally, let
\begin{equation*}
M_{p(\cdot)}^{\mathrm{rad}}(\Gamma)
:=
\inf\left\{
\int_{A(r_1,r_2)}
\rho(x)^{p(x)}\,dx:
\rho\in\mathcal F(\Gamma),
\ \rho\text{ is radial}
\right\}.
\end{equation*}
Since the radial admissible densities form a subclass of
$\mathcal F(\Gamma)$,
\begin{equation*}
M_{p(\cdot)}(\Gamma)
\le
M_{p(\cdot)}^{\mathrm{rad}}(\Gamma).
\end{equation*}
On the other hand, for every $\rho\in\mathcal F(\Gamma)$, the spherical
average $\widetilde\rho$ is radial, admissible, and satisfies
\begin{equation*}
\int_{A(r_1,r_2)}
\widetilde\rho^{\,p}\,dx
\le
\int_{A(r_1,r_2)}
\rho^{\,p}\,dx.
\end{equation*}
Taking the infimum over all $\rho\in\mathcal F(\Gamma)$ gives
\begin{equation*}
M_{p(\cdot)}^{\mathrm{rad}}(\Gamma)
\le
M_{p(\cdot)}(\Gamma).
\end{equation*}
Thus
\begin{equation*}
M_{p(\cdot)}^{\mathrm{rad}}(\Gamma)
=
M_{p(\cdot)}(\Gamma),
\end{equation*}
which proves \eqref{eq:radial_infimum}. We are done.
\end{proof}

\subsection{One-dimensional formulation and extremal density for the annulus}
The preceding radial reduction converts the original modulus problem into a
one-dimensional constrained minimization problem, whose existence and
uniqueness properties are summarized in the following theorem.

\begin{theorem}
\label{thm:1d_modulus}
Under the assumptions of Lemma~\ref{lem:radial_suffices}, the modular
variable exponent modulus admits the one-dimensional representation
\begin{equation}
\label{eq:annulus_1d_reduction}
M_{p(\cdot)}(\Gamma)
=
\inf_{\varphi\in\mathcal A_r} J_r(\varphi),
\end{equation}
where
\begin{equation}
\label{eq:annulus_admissible_class}
\mathcal A_r
=
\left\{
\varphi\in L^{q_0(\cdot)}(r_1,r_2):
\ \varphi\ge0\ \text{a.e.},\
\int_{r_1}^{r_2}\varphi(r)\,dr=1
\right\},
\end{equation}
and
\begin{equation}
\label{eq:annulus_1d_energy}
J_r(\varphi)
=
\omega_{n-1}
\int_{r_1}^{r_2}
\varphi(r)^{q_0(r)}r^{n-1}\,dr.
\end{equation}
Moreover, $J_r$ admits a unique minimizer in $\mathcal A_r$.
\end{theorem}

\begin{proof}
By Lemma~\ref{lem:radial_suffices}, it is sufficient to consider
radial admissible densities of the form
\begin{equation*}
\rho(x)=\varphi(|x|),
\qquad
\varphi\ge0.
\end{equation*}
For such a density, the polar-coordinate formula and the relation
$p(x)=q_0(|x|)$ give
\begin{equation}
\label{eq:radial_energy_identity}
\int_{A(r_1,r_2)}
\rho(x)^{p(x)}\,dx
=
\omega_{n-1}
\int_{r_1}^{r_2}
\varphi(r)^{q_0(r)}r^{n-1}\,dr
=
J_r(\varphi).
\end{equation}

We first identify precisely which radial profiles are admissible.
Suppose that
\begin{equation*}
\int_{r_1}^{r_2}\varphi(r)\,dr\ge1.
\end{equation*}
For any $\gamma\in\Gamma$, let $\mu(s)=|\gamma(s)|$. Since
$\mu$ is absolutely continuous and
$|\mu'(s)|\le|\gamma'(s)|$ for almost every $s$, Lemma~\ref{lem:projection}
yields
\begin{equation*}
\int_\gamma \rho\,ds
=
\int_0^T
\varphi(\mu(s))|\gamma'(s)|\,ds
\ge
\int_0^T
\varphi(\mu(s))|\mu'(s)|\,ds
\ge
\int_{r_1}^{r_2}\varphi(r)\,dr
\ge1.
\end{equation*}
Thus $\rho\in\mathcal F(\Gamma)$.

Conversely, suppose that $\rho(x)=\varphi(|x|)$ belongs to
$\mathcal F(\Gamma)$. Fix $\theta\in S^{n-1}$ and consider the radial
segment $\gamma_\theta(r)=r\theta,$ $r\in[r_1,r_2]$.
This curve belongs to $\Gamma$, and its arc-length parameter is precisely
the radial variable. Hence
\begin{equation*}
1
\le
\int_{\gamma_\theta}\rho\,ds
=
\int_{r_1}^{r_2}\varphi(r)\,dr.
\end{equation*}
Therefore,
\begin{equation}
\label{eq:radial_admissibility_characterization}
\rho(x)=\varphi(|x|)\in\mathcal F(\Gamma)
\quad\Longleftrightarrow\quad
\int_{r_1}^{r_2}\varphi(r)\,dr\ge1.
\end{equation}

Consequently, Lemma~\ref{lem:radial_suffices} and
\eqref{eq:radial_energy_identity} imply
\begin{equation}
\label{eq:annulus_inf_geq}
M_{p(\cdot)}(\Gamma)
=
\inf
\left\{
J_r(\varphi):
\varphi\ge0,\
\int_{r_1}^{r_2}\varphi(r)\,dr\ge1
\right\}.
\end{equation}

The constraint can be replaced by equality. Indeed, let $\varphi$ be
a nonnegative profile with finite $J_r(\varphi)$ and put
\begin{equation*}
c=\int_{r_1}^{r_2}\varphi(r)\,dr.
\end{equation*}
If $c>1$, define $\widehat\varphi=\varphi/c$. Then
$\int_{r_1}^{r_2}\widehat\varphi\,dr=1$, and
\begin{equation*}
J_r(\widehat\varphi)
=
\omega_{n-1}
\int_{r_1}^{r_2}
\varphi(r)^{q_0(r)}
c^{-q_0(r)}r^{n-1}\,dr.
\end{equation*}
Since $c>1$ and $q_0(r)>1$ for almost every $r$,
\begin{equation*}
c^{-q_0(r)}<1
\qquad\text{for a.e. }r\in(r_1,r_2),
\end{equation*}
and therefore $J_r(\widehat\varphi)\le J_r(\varphi).$
In fact, the inequality is strict whenever $J_r(\varphi)>0$.
Thus, profiles with integral strictly larger than $1$ cannot yield the
infimum. It follows from \eqref{eq:annulus_inf_geq} that
\begin{equation*}
M_{p(\cdot)}(\Gamma)
=
\inf_{\varphi\in\mathcal A_r}J_r(\varphi),
\end{equation*}
which proves \eqref{eq:annulus_1d_reduction}.

It remains to prove existence and uniqueness of the minimizer. Let
\begin{equation*}
m_r=\inf_{\varphi\in\mathcal A_r}J_r(\varphi),
\end{equation*}
and choose a minimizing sequence $(\varphi_k)\subset\mathcal A_r$ such
that
\begin{equation*}
J_r(\varphi_k)\longrightarrow m_r.
\end{equation*}
After discarding finitely many terms, we may assume that
$J_r(\varphi_k)\le M_0$ for some constant $M_0>0$. Since
$r\ge r_1>0$ on $(r_1,r_2)$,
\begin{equation*}
\int_{r_1}^{r_2}
\varphi_k(r)^{q_0(r)}\,dr
\le
\frac{M_0}{\omega_{n-1}r_1^{\,n-1}}.
\end{equation*}
The modular--norm comparison for variable exponent spaces therefore
shows that $(\varphi_k)$ is bounded in
$L^{q_0(\cdot)}(r_1,r_2)$. Since
\begin{equation*}
1<q_0^-\le q_0(r)\le q_0^+<\infty,
\end{equation*}
the space $L^{q_0(\cdot)}(r_1,r_2)$ is reflexive. Hence, after passing to
a subsequence, $\varphi_k\rightharpoonup\varphi_*$ weakly in $L^{q_0(\cdot)}(r_1,r_2)$
for some $\varphi_*\in L^{q_0(\cdot)}(r_1,r_2)$.

The admissible class $\mathcal A_r$ is convex. It is also strongly
closed because the embedding
\begin{equation*}
L^{q_0(\cdot)}(r_1,r_2)\hookrightarrow L^1(r_1,r_2)
\end{equation*}
is continuous under the present uniform lower bound $q_0^->1$.
Consequently, the constraint
$\int_{r_1}^{r_2}\varphi(r)\,dr=1$
is preserved under strong convergence, while nonnegativity is preserved
by passing to an almost-everywhere convergent subsequence. Since a
closed convex subset of a Banach space is weakly closed, it follows that
$\varphi_*\in\mathcal A_r$.

The functional $J_r$ is convex, since
$s\mapsto s^{q_0(r)}$ is strictly convex on $[0,\infty)$ for almost every
$r$, and the weight $r^{n-1}$ is bounded above and below by positive
constants on $[r_1,r_2]$. Moreover, the corresponding integral modular
is continuous with respect to the Luxemburg norm on
$L^{q_0(\cdot)}(r_1,r_2)$. Hence, $J_r$ is weakly lower semicontinuous.
Therefore,
\begin{equation*}
J_r(\varphi_*)
\le
\liminf_{k\to\infty}J_r(\varphi_k)
=
m_r.
\end{equation*}
Since $\varphi_*\in\mathcal A_r$, the reverse inequality
$J_r(\varphi_*)\ge m_r$ is immediate. Thus,
$J_r(\varphi_*)=m_r$,
and $\varphi_*$ is a minimizer.

Finally, suppose that $\varphi_1,\varphi_2\in\mathcal A_r$ are two
distinct minimizers. Then they differ on a measurable set of positive
measure. For almost every $r$, the function
$s\mapsto s^{q_0(r)}$ is strictly convex on $[0,\infty)$, and
$r^{n-1}>0$. Hence
\begin{equation*}
\left(\frac{\varphi_1(r)+\varphi_2(r)}{2}\right)^{q_0(r)}
<
\frac{\varphi_1(r)^{q_0(r)}
+\varphi_2(r)^{q_0(r)}}{2}
\end{equation*}
on a set of positive measure. Since
$(\varphi_1+\varphi_2)/2\in\mathcal A_r$, integration gives
\begin{equation*}
J_r\left(\frac{\varphi_1+\varphi_2}{2}\right)
<
\frac{J_r(\varphi_1)+J_r(\varphi_2)}{2}
=
m_r,
\end{equation*}
which contradicts the definition of $m_r$. Thus, the minimizer is unique.
\end{proof}

The one-dimensional variational problem can now be solved explicitly by deriving
the Euler--Lagrange equation and determining the associated normalization constant.
\begin{theorem}
\label{thm:annulus_formula}
Under the assumptions of Lemma~\ref{lem:radial_suffices}, let
$\rho_*$ denote the unique minimizer of $J_r$ in $\mathcal A_r$.
Then $\rho_*>0$ almost everywhere in $(r_1,r_2)$ and there exists a
unique constant $\lambda>0$ such that
\begin{equation}
\label{eq:annulus_EL}
\omega_{n-1}q_0(r)r^{n-1}
\rho_*(r)^{q_0(r)-1}
=
\lambda
\end{equation}
for almost every $r\in(r_1,r_2)$.
Consequently,
\begin{equation}
\label{eq:annulus_extr}
\rho_*(r)
=
\left(
\frac{\lambda}
{\omega_{n-1}q_0(r)r^{n-1}}
\right)^{\!1/(q_0(r)-1)}
\end{equation}
for almost every $r\in(r_1,r_2)$,
where $\lambda>0$ is uniquely determined by
\begin{equation}
\label{eq:annulus_normalization}
\int_{r_1}^{r_2}
\left(
\frac{\lambda}
{\omega_{n-1}q_0(r)r^{n-1}}
\right)^{\!1/(q_0(r)-1)}
\,dr
=
1.
\end{equation}
Moreover,
\begin{equation}
\label{eq:annulus_modulus_formula}
M_{p(\cdot)}(\Gamma)
=
\omega_{n-1}
\int_{r_1}^{r_2}
\rho_*(r)^{q_0(r)}r^{n-1}\,dr.
\end{equation}
\end{theorem}

\begin{proof}
Let $\rho_*$ be the unique minimizer of $J_r$ in $\mathcal A_r$, whose
existence and uniqueness were established in
Theorem~\ref{thm:1d_modulus}. We first note that
\begin{equation}
\label{eq:annulus_positive}
\rho_*(r)>0
\end{equation}
for almost every $r\in(r_1,r_2)$.
Indeed, the positivity argument used for the unweighted one-dimensional
problem applies without change to the present functional, because
\begin{equation*}
0<\omega_{n-1}r_1^{n-1}
\le
\omega_{n-1}r^{n-1}
\le
\omega_{n-1}r_2^{n-1}
<\infty
\end{equation*}
on $(r_1,r_2)$. If $\rho_*$ vanished on a set of positive measure, one
could transfer a sufficiently small amount of mass from a subset on
which $\rho_*$ is bounded below to that zero set. The contribution of
the zero set would be of order $\varepsilon^{q_0^-}$, whereas the
decrease on the positive set would be of order $\varepsilon$. Since
$q_0^->1$, this would strictly decrease $J_r$, contradicting minimality.

We next establish the Euler--Lagrange equation. Define
\begin{equation*}
g(r)
:=
\omega_{n-1}q_0(r)r^{n-1}
\rho_*(r)^{q_0(r)-1}.
\end{equation*}
Fix $0<\delta<N<\infty$ and consider the measurable level set
\begin{equation*}
E_{\delta,N}
:=
\{r\in(r_1,r_2):\delta\le\rho_*(r)\le N\}.
\end{equation*}
By \eqref{eq:annulus_positive}, these sets exhaust $(r_1,r_2)$ up to a
null set as $\delta\downarrow0$ and $N\uparrow\infty$.

Let $S,T\subset E_{\delta,N}$ be disjoint measurable sets of positive
measure. Define $h$ as in \eqref{fun-h}.
Then $\int_{r_1}^{r_2}h(r)\,dr=0.$
For sufficiently small $|\varepsilon|$, the function
$\rho_\varepsilon=\rho_*+\varepsilon h$
is nonnegative and satisfies
$\int_{r_1}^{r_2}\rho_\varepsilon(r)\,dr=1.$
Hence, $\rho_\varepsilon\in\mathcal A_r$ for both signs of sufficiently
small $\varepsilon$.

On $S\cup T$, the functions $\rho_\varepsilon$ remain in a fixed
compact subinterval of $(0,\infty)$ for sufficiently small
$|\varepsilon|$. Since $q_0$ is uniformly bounded between $q_0^-$ and
$q_0^+$, the difference quotients
\begin{equation*}
\frac{
1
}{\varepsilon}\left(\rho_\varepsilon(r)^{q_0(r)}
-
\rho_*(r)^{q_0(r)}\right)
\end{equation*}
are dominated by an integrable function independent of $\varepsilon$.
Therefore, differentiation under the integral sign is justified, and
minimality of $\rho_*$ for both positive and negative $\varepsilon$
gives
\begin{equation}
\label{eq:annulus_first_variation}
0
=
\left.
\frac{d}{d\varepsilon}
J_r(\rho_\varepsilon)
\right|_{\varepsilon=0}
=
\int_{r_1}^{r_2}
g(r)h(r)\,dr.
\end{equation}
Consequently,
\begin{equation*}
\frac{1}{|S|}\int_S g(r)\,dr
=
\frac{1}{|T|}\int_T g(r)\,dr.
\end{equation*}
Since this holds for every pair of disjoint positive-measure subsets
$S,T\subset E_{\delta,N}$, it follows that $g$ is constant almost
everywhere on $E_{\delta,N}$. Thus, there exists a constant
$\lambda_{\delta,N}$ such that $g(r)=\lambda_{\delta,N}$
for almost every $r\in E_{\delta,N}$.

If two such level sets have an intersection of positive measure, their
corresponding constants must coincide. Since the sets
$E_{\delta,N}$ exhaust $(r_1,r_2)$ up to a null set, there is therefore
a single constant $\lambda\in\mathbb R$ such that $g(r)=\lambda$ for almost every $r\in(r_1,r_2)$.
This proves \eqref{eq:annulus_EL}.

In view of
\eqref{eq:annulus_positive}, together with
$q_0(r)>1$ and $r>0$, every factor on the left-hand side of
\eqref{eq:annulus_EL} is positive almost everywhere. Hence
$\lambda>0$.
Solving \eqref{eq:annulus_EL} for $\rho_*(r)$ gives
\eqref{eq:annulus_extr}. Since $\rho_*\in\mathcal A_r$, its normalization
immediately yields \eqref{eq:annulus_normalization}.

It remains to prove that the normalization equation determines $\lambda$
uniquely. Define
\begin{equation}
\label{eq:annulus_G}
G(\lambda)
:=
\int_{r_1}^{r_2}
\left(
\frac{\lambda}
{\omega_{n-1}q_0(r)r^{n-1}}
\right)^{\!1/(q_0(r)-1)}
\,dr,
\qquad \lambda>0.
\end{equation}
For each fixed $r$, the integrand is continuous and strictly increasing
as a function of $\lambda>0$. Hence $G$ is strictly increasing.

We next verify continuity. Put
\begin{equation*}
d_-
:=
\omega_{n-1}q_0^-r_1^{n-1},
\qquad
d_+
:=
\omega_{n-1}q_0^+r_2^{n-1}.
\end{equation*}
More directly, since
\begin{equation*}
\omega_{n-1}q_0^-r_1^{n-1}
\le
\omega_{n-1}q_0(r)r^{n-1}
\le
\omega_{n-1}q_0^+r_2^{n-1},
\end{equation*}
the quotient appearing in \eqref{eq:annulus_G} is uniformly bounded
above on every compact interval $[\lambda_1,\lambda_2]\subset(0,\infty)$.
Indeed, if
\begin{equation*}
a_1=\frac{\lambda_1}
{\omega_{n-1}q_0^+r_2^{n-1}},
\qquad
a_2=\frac{\lambda_2}
{\omega_{n-1}q_0^-r_1^{n-1}},
\end{equation*}
then the quotient belongs to $[a_1,a_2]$, while
\begin{equation*}
\frac{1}{q_0^+-1}
\le
\frac{1}{q_0(r)-1}
\le
\frac{1}{q_0^--1}.
\end{equation*}
Thus, the integrand is bounded by a finite constant depending only on
$\lambda_2$, $q_0^\pm$, $r_1$, $r_2$, and $\omega_{n-1}$. Since
$(r_1,r_2)$ has finite measure, dominated convergence proves that
$G$ is continuous on $(0,\infty)$.

As $\lambda\downarrow0$, the integrand in \eqref{eq:annulus_G}
converges pointwise to zero. For $0<\lambda\le1$, it is bounded above by
the integrand corresponding to $\lambda=1$, which is integrable by the
same uniform bounds. Hence
$\lim_{\lambda\downarrow0}G(\lambda)=0.$
On the other hand, as $\lambda\uparrow\infty$, the integrand converges
pointwise to $+\infty$. By the monotone convergence theorem,
$\lim_{\lambda\uparrow\infty}G(\lambda)=+\infty.$
Therefore, the intermediate value theorem and strict monotonicity imply
that there exists a unique $\lambda>0$ satisfying $G(\lambda)=1$.
This proves the uniqueness asserted in \eqref{eq:annulus_normalization}.

Finally, by Theorem~\ref{thm:1d_modulus},
\begin{equation*}
M_{p(\cdot)}(\Gamma)
=
\inf_{\varphi\in\mathcal A_r}J_r(\varphi)
=
J_r(\rho_*),
\end{equation*}
and therefore
\begin{equation*}
M_{p(\cdot)}(\Gamma)
=
\omega_{n-1}
\int_{r_1}^{r_2}
\rho_*(r)^{q_0(r)}r^{n-1}\,dr.
\end{equation*}
This proves \eqref{eq:annulus_modulus_formula} and we are done.
\end{proof}

The explicit formula above immediately recovers the classical extremal
density and modulus in the constant-exponent case.
\begin{corollary}
\label{cor:constant_exp}
Suppose that $q_0(r)\equiv p,$ where $1<p<\infty$.
Then the unique extremal density is
\begin{equation}
\label{eq:constant_extremal_density}
\rho_*(r)
=
\frac{
r^{-\frac{n-1}{p-1}}
}{
\displaystyle
\int_{r_1}^{r_2}
s^{-\frac{n-1}{p-1}}\,ds
},
\end{equation}
and the modular modulus is
\begin{equation}
\label{eq:constant_modulus}
M_p(\Gamma)
=
\omega_{n-1}
\left(
\int_{r_1}^{r_2}
r^{-\frac{n-1}{p-1}}\,dr
\right)^{1-p}.
\end{equation}
\end{corollary}

\begin{proof}
When $q_0(r)\equiv p$, equation \eqref{eq:annulus_extr} becomes
\begin{equation*}
\rho_*(r)
=
\left(
\frac{\lambda}
{\omega_{n-1}p}
\right)^{1/(p-1)}
r^{-\frac{n-1}{p-1}}.
\end{equation*}
Thus
\begin{equation*}
\rho_*(r)=C r^{-\frac{n-1}{p-1}}
\end{equation*}
for a positive constant $C$. The normalization
\begin{equation*}
\int_{r_1}^{r_2}\rho_*(r)\,dr=1
\end{equation*}
gives
\begin{equation*}
C
=
\left(
\int_{r_1}^{r_2}
s^{-\frac{n-1}{p-1}}\,ds
\right)^{-1},
\end{equation*}
which proves \eqref{eq:constant_extremal_density}.

Substituting this expression into
\begin{equation*}
M_p(\Gamma)
=
\omega_{n-1}
\int_{r_1}^{r_2}
\rho_*(r)^p r^{n-1}\,dr
\end{equation*}
gives
\begin{align*}
M_p(\Gamma)
=
\omega_{n-1}C^p
\int_{r_1}^{r_2}
r^{-\frac{p(n-1)}{p-1}}r^{n-1}\,dr
=
\omega_{n-1}C^p
\int_{r_1}^{r_2}
r^{-\frac{n-1}{p-1}}\,dr.
\end{align*}
If
\begin{equation*}
I=
\int_{r_1}^{r_2}
r^{-\frac{n-1}{p-1}}\,dr,
\end{equation*}
then $C=I^{-1}$, and hence
\begin{equation*}
M_p(\Gamma)
=
\omega_{n-1}I^{-p}I
=
\omega_{n-1}I^{1-p}.
\end{equation*}
This is precisely \eqref{eq:constant_modulus}.
\end{proof}

\subsection{Reduction for the cylindrical domain}
\label{sec:cylinder}
The following reduction requires only that the cross-sectional domain
have finite positive measure; in particular, no regularity of its
boundary is needed.
\begin{proposition}
\label{prop:cylinder}
Let $D\subset\mathbb R^{n-1}$ be a domain with
$0<A:=|D|<\infty$,
and let
\begin{equation*}
\mathcal C:=D\times(0,L),
\qquad L>0.
\end{equation*}
Suppose that the variable exponent depends only on the axial
coordinate, namely
$p(x',t)=\eta(t)$,
where $\eta:(0,L)\to(1,\infty)$ is measurable and satisfies
\begin{equation}
\label{eq:cylinder_exp_bounds}
1<\eta^-\le \eta(t)\le \eta^+<\infty
\end{equation}
for almost every $t\in(0,L)$.
Let $\Gamma$ be the family of all locally rectifiable curves
$\gamma:[0,T]\longrightarrow\overline{\mathcal C}$
whose initial and terminal points belong to
$D\times\{0\}$ and $D\times\{L\}$, respectively, and whose interior
lies in $\mathcal C$.
Then
\begin{equation}
\label{eq:cylinder_1d_reduction}
M_{p(\cdot)}(\Gamma)
=
\inf_{\varphi\in\mathcal A_c}
A\int_0^L \varphi(t)^{\eta(t)}\,dt,
\end{equation}
where
\begin{equation}
\label{eq:cylinder_admissible_class}
\mathcal A_c
=
\left\{
\varphi\in L^{\eta(\cdot)}(0,L):
\ \varphi\ge0\ \text{a.e. on }(0,L),
\ \int_0^L\varphi(t)\,dt=1
\right\}.
\end{equation}
\end{proposition}

\begin{proof}
We first establish the inequality
\begin{equation}
\label{eq:cylinder_lower_bound}
\inf_{\varphi\in\mathcal A_c}
A\int_0^L\varphi(t)^{\eta(t)}\,dt
\le
M_{p(\cdot)}(\Gamma).
\end{equation}
Let $\rho\in\mathcal F(\Gamma)$ be an admissible density with finite
modular energy,
\begin{equation*}
\int_{\mathcal C}\rho(x)^{p(x)}\,dx<\infty.
\end{equation*}
Define its fiber average by
\begin{equation}
\label{eq:cylinder_fiber_average}
\overline\rho(t)
=
\frac{1}{A}\int_D\rho(x',t)\,dx',
\qquad t\in(0,L).
\end{equation}
For every $x'\in D$, the vertical segment
\begin{equation*}
\gamma_{x'}(t)=(x',t),
\qquad t\in[0,L],
\end{equation*}
belongs to $\Gamma$. Hence the admissibility of $\rho$ gives
\begin{equation*}
\int_0^L\rho(x',t)\,dt\ge1
\end{equation*}
for every $x'\in D$.
Integrating over $D$ with respect to the normalized measure
$dx'/A$ and applying Tonelli's theorem yields
\begin{align}
1
\le
\frac1A\int_D
\left(\int_0^L\rho(x',t)\,dt\right)dx'
=
\int_0^L
\left(\frac1A\int_D\rho(x',t)\,dx'\right)dt
=
\int_0^L\overline\rho(t)\,dt.
\label{eq:cylinder_average_constraint}
\end{align}
Set
\begin{equation*}
c:=\int_0^L\overline\rho(t)\,dt.
\end{equation*}
Thus, $c\ge1$.
Since $\eta(t)>1$ for almost every $t$, Jensen's inequality applied
to the convex function $s\mapsto s^{\eta(t)}$ with respect to the
probability measure $dx'/A$ gives
\begin{equation}
\label{eq:cylinder_jensen}
\overline\rho(t)^{\eta(t)}
\le
\frac1A\int_D
\rho(x',t)^{\eta(t)}\,dx'
\end{equation}
for almost every $t\in(0,L)$.
Consequently,
\begin{align}
A\int_0^L\overline\rho(t)^{\eta(t)}\,dt
\le
\int_0^L\int_D
\rho(x',t)^{\eta(t)}\,dx'\,dt
=
\int_{\mathcal C}\rho(x)^{p(x)}\,dx
<\infty.
\label{eq:cylinder_energy_average}
\end{align}
In particular, we have
\begin{equation*}
\overline\rho\in L^{\eta(\cdot)}(0,L).
\end{equation*}

If $c=1$, we put
$\psi=\overline\rho$.
If $c>1$, define
$\psi=\overline\rho/c$.
In either case, $\psi\in\mathcal A_c$.
Moreover, if $c>1$, then $c^{-\eta(t)}<1$
for almost every $t\in(0,L)$,
and therefore
\begin{equation}
\label{eq:cylinder_rescaling}
A\int_0^L\psi(t)^{\eta(t)}\,dt
\le
A\int_0^L\overline\rho(t)^{\eta(t)}\,dt.
\end{equation}
Combining \eqref{eq:cylinder_energy_average} and
\eqref{eq:cylinder_rescaling}, we obtain
\begin{equation*}
\inf_{\varphi\in\mathcal A_c}
A\int_0^L\varphi(t)^{\eta(t)}\,dt
\le
\int_{\mathcal C}\rho(x)^{p(x)}\,dx.
\end{equation*}
Taking the infimum over all admissible densities of finite modular
energy gives \eqref{eq:cylinder_lower_bound}.

Next we prove the reverse inequality. Let
$\varphi\in\mathcal A_c$ and define the axial density
\begin{equation}
\label{eq:cylinder_axial_density}
\rho(x',t)=\varphi(t).
\end{equation}
Let $\gamma:[0,T]\longrightarrow\overline{\mathcal C}$
be any curve in $\Gamma$, parametrized by arc length, and write
\begin{equation*}
\gamma(s)=(x'(s),t(s)).
\end{equation*}
The coordinate projection onto the last component is $1$-Lipschitz;
hence $|t'(s)|\le|\gamma'(s)|$ for almost every $s\in(0,T)$.
Furthermore, $t(0)=0$ and $t(T)=L$.
Therefore, by the one-dimensional projection inequality
(Lemma~\ref{lem:projection}), we obtain
\begin{align}
\int_\gamma\rho\,ds
=
\int_0^T
\varphi(t(s))|\gamma'(s)|\,ds
\ge
\int_0^T
\varphi(t(s))|t'(s)|\,ds
\ge
\int_0^L\varphi(t)\,dt
=
1.
\label{eq:cylinder_axial_admissibility}
\end{align}
Thus, $\rho\in\mathcal F(\Gamma)$.

Since $p(x',t)=\eta(t)$ and $\rho(x',t)=\varphi(t)$,
Fubini's theorem gives
\begin{align}
\int_{\mathcal C}\rho(x)^{p(x)}\,dx
=
\int_0^L
\int_D
\varphi(t)^{\eta(t)}\,dx'\,dt
=
A\int_0^L
\varphi(t)^{\eta(t)}\,dt.
\label{eq:cylinder_axial_energy}
\end{align}
Consequently,
\begin{equation*}
M_{p(\cdot)}(\Gamma)
\le
A\int_0^L\varphi(t)^{\eta(t)}\,dt.
\end{equation*}
Taking the infimum over $\varphi\in\mathcal A_c$ yields
\begin{equation}
\label{eq:cylinder_upper_bound}
M_{p(\cdot)}(\Gamma)
\le
\inf_{\varphi\in\mathcal A_c}
A\int_0^L\varphi(t)^{\eta(t)}\,dt.
\end{equation}
Combining \eqref{eq:cylinder_lower_bound} and
\eqref{eq:cylinder_upper_bound} proves \eqref{eq:cylinder_1d_reduction}. The proof is now complete.
\end{proof}

\section{Test Densities and Upper Bounds}
\label{sec:bounds}

\subsection{Upper bound for the annulus}

The explicit extremal formula in Theorem~\ref{thm:annulus_formula}
also provides a natural benchmark for simpler test densities. In
particular, the classical logarithmic density is admissible for the
annular curve family and yields an explicit upper bound for the
variable exponent modulus. The following theorem also identifies the
case in which this classical density remains exactly extremal.

\begin{theorem}
\label{thm:annulus_bounds}
Under the assumptions of Theorem~\ref{thm:annulus_formula}, let $p(x)=q_0(|x|)$,
where
\begin{equation*}
1<q_0^-\le q_0(r)\le q_0^+<\infty
\end{equation*}
for almost every $r\in(r_1,r_2)$.
Define the logarithmic radial profile by
\begin{equation}
\label{eq:log_test_profile}
\rho_{\log}(r)
=
\frac{1}{r\log(r_2/r_1)},
\qquad r\in(r_1,r_2),
\end{equation}
and the corresponding radial density on $A(r_1,r_2)$ by
$x\longmapsto \rho_{\log}(|x|).$
Then this density belongs to $\mathcal F(\Gamma)$. Consequently,
\begin{equation}
\label{eq:annulus_log_upper_bound}
M_{p(\cdot)}(\Gamma)
\le
\omega_{n-1}
\int_{r_1}^{r_2}
\frac{r^{n-1-q_0(r)}}
{\bigl[\log(r_2/r_1)\bigr]^{q_0(r)}}
\,dr.
\end{equation}

Moreover, if $q_0(r)\equiv n,$
then $\rho_{\log}$ is the unique extremal radial profile and
\begin{equation}
\label{eq:annulus_log_equality}
M_{p(\cdot)}(\Gamma)
=
\omega_{n-1}
\int_{r_1}^{r_2}
\frac{dr}
{r\,[\log(r_2/r_1)]^n}
=
\omega_{n-1}
\biggl(\log\frac{r_2}{r_1}\biggr)^{1-n}.
\end{equation}
\end{theorem}

\begin{proof}
We first verify admissibility. Let
$\gamma:[0,T]\longrightarrow\overline{A(r_1,r_2)}$
be a locally rectifiable curve in $\Gamma$, parametrized by arc length,
with its interior contained in $A(r_1,r_2)$ and with $|\gamma(0)|=r_1,$ and $|\gamma(T)|=r_2.$
Set $\mu(s)=|\gamma(s)|.$
Since the map $x\mapsto |x|$ is $1$-Lipschitz,
$\mu$ is absolutely continuous and $|\mu'(s)|\le |\gamma'(s)|$ for almost every $s\in(0,T)$.
Moreover, $\mu(0)=r_1,$ and $\mu(T)=r_2.$
Therefore,
\begin{align}
\int_\gamma \rho_{\log}\,ds
=
\frac{1}{\log(r_2/r_1)}
\int_0^T
\frac{|\gamma'(s)|}{\mu(s)}\,ds
\ge
\frac{1}{\log(r_2/r_1)}
\int_0^T
\frac{|\mu'(s)|}{\mu(s)}\,ds.
\label{eq:log_density_projection}
\end{align}
Applying Lemma~\ref{lem:projection} to the nonnegative Borel function $r\longmapsto\ 1/r$
gives
\begin{align}
\int_\gamma \rho_{\log}\,ds
\ge
\frac{1}{\log(r_2/r_1)}
\int_{r_1}^{r_2}\frac{dr}{r}
=1.
\label{eq:log_density_admissible}
\end{align}
Hence the radial density induced by $\rho_{\log}$ belongs to
$\mathcal F(\Gamma)$.

By the polar-coordinate formula and the relation
$p(x)=q_0(|x|)$, we obtain
\begin{align}
M_{p(\cdot)}(\Gamma)
\le
\omega_{n-1}
\int_{r_1}^{r_2}
\rho_{\log}(r)^{q_0(r)}r^{n-1}\,dr
=
\omega_{n-1}
\int_{r_1}^{r_2}
\frac{r^{n-1-q_0(r)}}
{\bigl[\log(r_2/r_1)\bigr]^{q_0(r)}}
\,dr.
\end{align}
This proves \eqref{eq:annulus_log_upper_bound}.

It remains to consider the case $q_0\equiv n$. By
Theorem~\ref{thm:annulus_formula}, the unique extremal profile
satisfies
\begin{equation*}
\rho_*(r)
=
\left(
\frac{\lambda}
{\omega_{n-1}n r^{n-1}}
\right)^{1/(n-1)}
=
C r^{-1},
\end{equation*}
where
\begin{equation*}
C=
\left(
\frac{\lambda}{\omega_{n-1}n}
\right)^{1/(n-1)}
>0.
\end{equation*}
The normalization condition gives
\begin{equation*}
1
=
\int_{r_1}^{r_2}\rho_*(r)\,dr
=
C\int_{r_1}^{r_2}\frac{dr}{r}
=
C\log\frac{r_2}{r_1}.
\end{equation*}
Hence
\begin{equation*}
C=\frac{1}{\log(r_2/r_1)},
\end{equation*}
and therefore
$\rho_*(r)=\rho_{\log}(r)$.
Since $\rho_*$ is the unique minimizer of the reduced functional,
the upper bound in \eqref{eq:annulus_log_upper_bound} is attained.
More explicitly,
\begin{align}
M_{p(\cdot)}(\Gamma)
=
\omega_{n-1}
\int_{r_1}^{r_2}
\left(
\frac{1}{r\log(r_2/r_1)}
\right)^n
r^{n-1}\,dr
=
\frac{\omega_{n-1}}
{\bigl[\log(r_2/r_1)\bigr]^n}
\int_{r_1}^{r_2}\frac{dr}{r}
=
\omega_{n-1}
\biggl(\log\frac{r_2}{r_1}\biggr)^{1-n}.
\end{align}
This proves \eqref{eq:annulus_log_equality}, and we are done.
\end{proof}

\subsection{Upper bound for the cylinder}

The one-dimensional reduction in Proposition~\ref{prop:cylinder}
provides a direct way to obtain explicit upper bounds by choosing
simple admissible axial profiles. The constant profile
$\varphi(t)=1/L$ is the natural analogue of the classical logarithmic
test density for the annulus. The following theorem gives the resulting
upper bounds and identifies the constant-exponent case in which this
test profile is exactly extremal.

\begin{theorem}
\label{thm:cylinder_bounds}
Under the assumptions of Proposition~\ref{prop:cylinder}, the modular
variable exponent modulus satisfies
\begin{equation}
\label{eq:cylinder_constant_upper_bound}
M_{p(\cdot)}(\Gamma)
\le
A\int_0^L L^{-\eta(t)}\,dt.
\end{equation}
Consequently,
\begin{equation}
\label{eq:cylinder_upper_bound_large_L}
M_{p(\cdot)}(\Gamma)
\le
A L^{1-\eta^-},
\qquad L\ge1,
\end{equation}
and
\begin{equation}
\label{eq:cylinder_upper_bound_small_L}
M_{p(\cdot)}(\Gamma)
\le
A L^{1-\eta^+},
\qquad 0<L\le1.
\end{equation}

If the exponent is constant,
$\eta(t)\equiv p_0\in(1,\infty),$
then the unique extremal profile in $\mathcal A_c$ is
\begin{equation}
\label{eq:cylinder_constant_extremal}
\varphi_*(t)=\frac1L,
\qquad 0<t<L.
\end{equation}
In particular,
\begin{equation}
\label{eq:cylinder_constant_modulus}
M_{p_0}(\Gamma)
=
A L^{1-p_0}.
\end{equation}
\end{theorem}

\begin{proof}
Consider the constant profile
\begin{equation*}
\varphi_0(t)=\frac1L,
\qquad 0<t<L.
\end{equation*}
Clearly,
\begin{equation*}
\varphi_0\ge0
\qquad\text{and}\qquad
\int_0^L\varphi_0(t)\,dt=1,
\end{equation*}
so $\varphi_0\in\mathcal A_c$. Proposition~\ref{prop:cylinder}
therefore gives
\begin{align}
M_{p(\cdot)}(\Gamma)
\le
A\int_0^L
\left(\frac1L\right)^{\eta(t)}\,dt
=
A\int_0^L L^{-\eta(t)}\,dt,
\end{align}
which proves \eqref{eq:cylinder_constant_upper_bound}.

Suppose first that $L\ge1$. Since $\log L\ge0$ and $\eta(t)\ge\eta^-$,
we have
\begin{equation*}
-\eta(t)\log L
\le
-\eta^-\log L.
\end{equation*}
Hence $L^{-\eta(t)}
\le
L^{-\eta^-}$
for almost every $t\in(0,L)$.
It follows that
\begin{equation*}
A\int_0^L L^{-\eta(t)}\,dt
\le
A\int_0^L L^{-\eta^-}\,dt
=
A L^{1-\eta^-},
\end{equation*}
proving \eqref{eq:cylinder_upper_bound_large_L}.

Now suppose that $0<L\le1$. Since $\log L\le0$ and
$\eta(t)\le\eta^+$,
multiplication by the nonpositive number $-\log L$ gives
\begin{equation*}
-\eta(t)\log L
\le
-\eta^+\log L.
\end{equation*}
Thus, $L^{-\eta(t)}
\le
L^{-\eta^+}$ for almost every $t\in(0,L),$
and consequently
\begin{equation*}
A\int_0^L L^{-\eta(t)}\,dt
\le
A\int_0^L L^{-\eta^+}\,dt
=
A L^{1-\eta^+}.
\end{equation*}
This proves \eqref{eq:cylinder_upper_bound_small_L}.

Finally, suppose that
$\eta(t)\equiv p_0$.
The reduced functional is then
\begin{equation*}
J_c(\varphi)
=
A\int_0^L\varphi(t)^{p_0}\,dt,
\qquad
\varphi\in\mathcal A_c.
\end{equation*}
By the Euler--Lagrange characterization applied to the one-dimensional
problem, every minimizer satisfies $p_0 A\,\varphi_*(t)^{p_0-1}
=
\lambda$ for almost every $t\in(0,L)$
for some constant $\lambda>0$. Hence $\varphi_*$ is constant almost
everywhere. The normalization condition then forces $\varphi_*(t)=1/L$ for almost every $t\in(0,L).$

The uniqueness of the minimizer follows from Theorem~\ref{thm:existence}, since
$\mathcal A_c$ coincides with the admissible class $\mathcal A$ of
Definition~\ref{def:1d_problem} (with $p=\eta\equiv p_0$), and $J_c=A\,\mathcal J$
is a positive constant multiple of the functional $\mathcal J$ studied there;
scaling by $A>0$ does not change the location of the minimizer. Therefore, the corresponding density
\begin{equation*}
\rho_*(x',t)=\frac1L
\end{equation*}
is the unique extremal density, up to equality almost everywhere.
Finally,
\begin{align}
M_{p_0}(\Gamma)
=
A\int_0^L
\left(\frac1L\right)^{p_0}\,dt
=
A L^{1-p_0},
\end{align}
which proves \eqref{eq:cylinder_constant_modulus}.
\end{proof}

\section{Capacity--Modulus Comparison}
\label{sec:capacity}

In this section, we relate the variable exponent relative capacity of a
condenser to the corresponding variable exponent modulus of the family of
curves joining its plates.

Throughout this section, let $\Omega\subset\mathbb{R}^{n}$ be a bounded
domain, and assume that $p\in\mathcal P^{\log}(\Omega),$ where
\begin{equation*}
1<p^-\le p(x)\le p^+<\infty,
\end{equation*}
for almost every $x\in\Omega$.
Let $E,F\Subset\Omega$ be disjoint nonempty compact sets, and let
$\Gamma=\Gamma(E,F;\Omega)$
denote the family of all locally rectifiable curves in $\Omega$ joining
$E$ to $F$.

We use the standard variable exponent relative capacity
\begin{equation*}
\operatorname{Cap}_{p(\cdot)}(E,F;\Omega)
=
\inf_{u\in\mathcal A(E,F;\Omega)}
\int_\Omega |\nabla u(x)|^{p(x)}\,dx,
\end{equation*}
where $\mathcal A(E,F;\Omega)$ denotes the class of admissible
Sobolev functions, with the boundary conditions
\begin{equation*}
u\ge1\quad\text{on }E,
\qquad
u\le0\quad\text{on }F,
\end{equation*}
understood in terms of the appropriate quasicontinuous representatives.

The preceding variational formulation identifies the curve modulus
with the natural energy minimization problem, and the corresponding
connection with condenser capacity is given by the following result.

Note that without the log-H\"older continuity assumption, assuming only the density of smooth functions yields a two-sided estimate with a dimensional distortion factor $n^{p^+/2}$ rather than an exact equality (see \cite[Theorem~5.4]{Har07} for the case of a single compact set relative to the boundary).

\begin{theorem}
\label{thm:cap_mod}
Under the assumptions above,
\begin{equation}
\label{eq:cap_mod_equality}
\operatorname{Cap}_{p(\cdot)}(E,F;\Omega)
=
M_{p(\cdot)}(\Gamma).
\end{equation}
\end{theorem}

\begin{proof}
We prove the two inequalities separately. We first show that
\begin{equation*}
\operatorname{Cap}_{p(\cdot)}(E,F;\Omega) \le M_{p(\cdot)}(\Gamma).
\end{equation*}
Let $\rho \in \mathcal{F}(\Gamma)$ be an admissible density with $\int_\Omega \rho(x)^{p(x)}\,dx < \infty$. Since $p^- > 1$ and $\Omega$ is bounded, $\rho \in L^{p(\cdot)}(\Omega) \subset L^1(\Omega)$.

Define the $\rho$-length distance from a point $x \in \Omega$ to $F$ by
\begin{equation*}
d_\rho(x,F) := \inf_{\gamma} \int_\gamma \rho\,ds,
\end{equation*}
where the infimum is taken over all locally rectifiable curves $\gamma \subset \Omega$ joining $x$ to a point of $F$. Set $u(x) := \min\{1, d_\rho(x,F)\}$.

By standard upper gradient construction, $\rho$ is an upper gradient of $u$ with $0 \le u \le 1$, $u=0$ on $F$, and $u=1$ on $E$. Since $u$ is bounded and $\Omega$ has finite measure, $u \in L^{p(\cdot)}(\Omega)$. Because its upper gradient $\rho$ is also in $L^{p(\cdot)}(\Omega)$, it follows that $u \in W^{1,p(\cdot)}(\Omega)$ with $|\nabla u(x)| \le \rho(x)$ for almost every $x \in \Omega$. Understanding the boundary values via its quasicontinuous representative, $u$ is a valid competitor for the relative capacity. Thus,
\begin{equation*}
\operatorname{Cap}_{p(\cdot)}(E,F;\Omega) \le \int_\Omega |\nabla u(x)|^{p(x)}\,dx \le \int_\Omega \rho(x)^{p(x)}\,dx.
\end{equation*}
Taking the infimum over all $\rho \in \mathcal{F}(\Gamma)$ yields $\operatorname{Cap}_{p(\cdot)}(E,F;\Omega) \le M_{p(\cdot)}(\Gamma)$.

To prove the reverse inequality, let $u \in \mathcal{A}(E,F;\Omega)$ be an admissible Sobolev function. By truncation, we assume $0 \le u \le 1$ in $\Omega$. For its quasicontinuous representative, there exist exceptional sets $E_0 \subset E$ and $F_0 \subset F$ of $p(\cdot)$-relative capacity zero such that $u=1$ on $E \setminus E_0$ and $u=0$ on $F \setminus F_0$.

Let $\Gamma_{pts}$ be the family of all rectifiable curves in $\Omega$ that intersect either $E_0$ or $F_0$. Since $\operatorname{Cap}_{p(\cdot)}(E_0 \cup F_0, \Omega) = 0$, the corresponding curve modulus satisfies $M_{p(\cdot)}(\Gamma_{pts}) = 0$. By the variable-exponent Fuglede-type theorem (see \cite{HHM}), the inequality
\begin{equation}\label{ineq: ACC}
|u(\gamma(0))-u(\gamma(T))| \le \int_\gamma |\nabla u|\,ds
\end{equation}
holds for every curve outside a family $\Gamma _{AC}$ of curves along which $u$ fails to be absolutely continuous, with $M_{p(\cdot)}(\Gamma_{AC}) = 0$.

Define the total exceptional curve family $\Gamma_0 = \Gamma_{pts} \cup \Gamma_{AC}$. By subadditivity, $M_{p(\cdot)}(\Gamma_0) = 0$. For any curve $\gamma \in \Gamma \setminus \Gamma_0$, its endpoints satisfy $\gamma(0) \in E \setminus E_0$ and $\gamma(T) \in F \setminus F_0$. Therefore,
\begin{equation*}
1 = |1 - 0| = |u(\gamma(0)) - u(\gamma(T))| \le \int_\gamma |\nabla u|\,ds.
\end{equation*}
Fix $\varepsilon > 0$ and choose an admissible density $\rho_\varepsilon \in \mathcal{F}(\Gamma_0)$ such that $\int_\Omega \rho_\varepsilon(x)^{p(x)}\,dx < \varepsilon$. It follows that $|\nabla u| + \rho_\varepsilon \in \mathcal{F}(\Gamma)$ because $\int_\gamma |\nabla u|\,ds \ge 1$ for $\gamma \notin \Gamma_0$, and $\int_\gamma \rho_\varepsilon\,ds \ge 1$ for $\gamma \in \Gamma_0$. Hence,
\begin{equation*}
M_{p(\cdot)}(\Gamma) \le \int_\Omega (|\nabla u(x)| + \rho_\varepsilon(x))^{p(x)}\,dx.
\end{equation*}
Letting $\varepsilon \downarrow 0$, the dominated convergence theorem implies the right-hand side converges to $\int_\Omega |\nabla u(x)|^{p(x)}\,dx$ since $\rho_\varepsilon \to 0$ in $L^{p(\cdot)}(\Omega)$. Thus, $M_{p(\cdot)}(\Gamma) \le \int_\Omega |\nabla u(x)|^{p(x)}\,dx$. Taking the infimum over all $u \in \mathcal{A}(E,F;\Omega)$ yields $M_{p(\cdot)}(\Gamma) \le \operatorname{Cap}_{p(\cdot)}(E,F;\Omega)$, completing the proof.
\end{proof}

\section{Remark on Quasiconformal Distortion}

The classical $n$-modulus is quasi-invariant under $K$-quasiconformal mappings, with the distortion being entirely controlled by $K$. However, for a variable exponent, the situation is fundamentally different when $p(\cdot)\not\equiv n$. The change-of-variables formula introduces an additional Jacobian factor, and even conformal similarities demonstrate that a comparison with a constant depending solely on $n$, $K$, and the bounds of the exponent cannot hold in general.

Let $f:\Omega\to\Omega'$ be a $K$-quasiconformal mapping between domains in $\mathbb R^n$, and define the pulled-back exponent by
\begin{equation*}
\widetilde p(y)=p(f^{-1}(y)), \qquad y\in\Omega'.
\end{equation*}
If $\rho$ is an admissible density for a curve family $\Gamma$ in $\Omega$, a natural valid admissible density $\widetilde\rho$ for the image family $f(\Gamma)$ is given by
\begin{equation*}
\widetilde\rho(y) = \rho(f^{-1}(y)) \|Df^{-1}(y)\|.
\end{equation*}
Consequently, the change-of-variables formula yields the exact pullback energy integral:
\begin{equation*}
\int_{\Omega'} \widetilde\rho(y)^{\widetilde p(y)}\,dy = \int_{\Omega} \rho(x)^{p(x)} \|Df^{-1}(f(x))\|^{p(x)} J_f(x)\,dx.
\end{equation*}
Utilizing the fundamental quasiconformal distortion inequality $\|Df^{-1}(f(x))\|^n J_f(x)\le K$ for almost every $x\in\Omega$, we can factor out the scaling terms:
\begin{equation*}
\|Df^{-1}(f(x))\|^{p(x)}J_f(x) = \bigl( \|Df^{-1}(f(x))\|^nJ_f(x) \bigr)^{p(x)/n} J_f(x)^{1-p(x)/n},
\end{equation*}
which directly produces the upper bound
\begin{equation*}
\|Df^{-1}(f(x))\|^{p(x)}J_f(x) \le K^{p(x)/n} J_f(x)^{1-p(x)/n}.
\end{equation*}
Thus, unless $p(x)=n$ almost everywhere, the change of variables leaves a nontrivial, uncompensated Jacobian weight.

This observation rigorously shows that a uniform comparison depending only on $n$, $K$, and the global bounds $p^-,p^+$ cannot hold without an additional normalization. Consider the conformal similarity scaling map:
\begin{equation*}
f(x)=a x, \qquad a>0.
\end{equation*}
This mapping is $1$-quasiconformal. However, for a constant exponent $p(x)\equiv q$, the standard modulus transforms via:
\begin{equation*}
M_q(f(\Gamma)) = a^{n-q} M_q(\Gamma).
\end{equation*}
Therefore, if $q\ne n$, the ratio $M_q(f(\Gamma))/M_q(\Gamma)$ can be made arbitrarily large or small by varying $a$. Consequently, no constant depending only on $n$, $K$, $q^-$, and $q^+$ can provide a two-sided modular comparison for arbitrary quasiconformal mappings.

The underlying difficulty in the variable exponent case is therefore not merely the higher integrability of the Jacobian. When $p(x)<n$, the factor $J_f(x)^{1-p(x)/n}$ involves a positive power of the Jacobian. In contrast, when $p(x)>n$, it involves a negative power, requiring explicit information on the compression of the mapping or, equivalently, quantitative estimates involving the inverse mapping. Standard higher integrability of $J_f$ alone is insufficient to yield the uniform modular bounds required for an arbitrary choice of admissible density.

\medskip
\noindent
\textbf{A natural open problem.}
The preceding obstruction suggests that a meaningful quasiconformal distortion theory for the variable exponent modulus must incorporate an appropriate normalization of the mapping or precise quantitative control of both its Jacobian and inverse Jacobian.

A natural problem is to determine sharp conditions, expressed in terms of $p\in\mathcal P^{\log}(\Omega)$ together with local higher-integrability bounds on the Jacobians ($J_f\in L^{1+\delta}_{\mathrm{loc}}(\Omega)$ and $J_{f^{-1}}\in L^{1+\delta}_{\mathrm{loc}}(\Omega')$ for some $\delta=\delta(n,K)>0$), under which the uniform two-sided inequality
\begin{equation}
\label{eq:variable_qc_comparison}
\frac{1}{C} M_{p(\cdot)}(\Gamma) \le M_{\widetilde p(\cdot)}(f(\Gamma)) \le C M_{p(\cdot)}(\Gamma)
\end{equation}
holds for all curve families $\Gamma$ in $\Omega$, where $\widetilde p=p\circ f^{-1}$. Specifically, it remains to identify the sharp dependence of the structural constant $C$ on the local oscillation of $p(\cdot)$, its log-H\"older regularity constants, and the explicit distortion parameters of $f$.


\section{Consequences}
\label{sec:applications}

\subsection{Monotonicity with respect to the outer radius}

The explicit one-dimensional formulation also yields a useful monotonicity property. In this statement, the inner radius is fixed and the radial exponent is regarded as a fixed profile on the corresponding range of radii.

The following corollary records the monotonicity of the modulus with respect to the outer radius and clarifies the corresponding thin- and thick-annulus asymptotics.

\begin{corollary}
\label{cor:monotonicity}
Fix $r_1>0$, and suppose that $q_0:[r_1,\infty)\longrightarrow(1,\infty)$ is measurable and satisfies
\begin{equation*}
1<q_0^-\le q_0(r)\le q_0^+<\infty
\end{equation*}
for almost every $r\ge r_1$. For $R>r_1$, let $A_R=A(r_1,R)$ and let $\Gamma_R$ denote the family of locally rectifiable curves in $A_R$ joining $S_{r_1}^{n-1}$ to $S_R^{n-1}$, with exponent $p(x)=q_0(|x|).$ Then the function
\begin{equation*}
R\longmapsto M_{p(\cdot)}(\Gamma_R)
\end{equation*}
is strictly decreasing on $(r_1,\infty)$. In particular,
\begin{equation*}
R_1<R_2 \quad\Longrightarrow\quad M_{p(\cdot)}(\Gamma_{R_2}) < M_{p(\cdot)}(\Gamma_{R_1}).
\end{equation*}
Moreover,
\begin{equation}
\label{eq:thin_annulus_limit}
M_{p(\cdot)}(\Gamma_R)\longrightarrow\infty \qquad\text{as }R\downarrow r_1.
\end{equation}
In general, no universal limit as $R\to\infty$ follows from the assumptions above.

If, in addition, $q_0(r)\ge n$ for almost every sufficiently large $r$, then
\begin{equation}
\label{eq:thick_annulus_limit}
M_{p(\cdot)}(\Gamma_R)\longrightarrow0 \qquad\text{as }R\to\infty.
\end{equation}
\end{corollary}

\begin{proof}
Let $R_1<R_2,$ and denote by $\varphi_{R_1}$ and $\varphi_{R_2}$ the unique minimizers in $\mathcal A_{R_1}$ and $\mathcal A_{R_2}$, respectively. Thus,
\begin{equation*}
\int_{r_1}^{R_1}\varphi_{R_1}(r)\,dr=1
\end{equation*}
and
\begin{equation*}
M_{p(\cdot)}(\Gamma_{R_1}) = \omega_{n-1} \int_{r_1}^{R_1} \varphi_{R_1}(r)^{q_0(r)}r^{n-1}\,dr.
\end{equation*}
We first prove strict monotonicity. Let $\varphi_{R_1}\in\mathcal A_{R_1}$ denote the unique minimizer of $J_r$ over $\mathcal A_{R_1}$, and extend it by zero to $(r_1,R_2)$, setting $\widehat\varphi(r)=\varphi_{R_1}(r)$ for $r\in(r_1,R_1)$ and $\widehat\varphi(r)=0$ for $r\in[R_1,R_2)$. Then $\widehat\varphi\ge0$ and $\int_{r_1}^{R_2}\widehat\varphi(r)\,dr =\int_{r_1}^{R_1}\varphi_{R_1}(r)\,dr=1$, so $\widehat\varphi\in\mathcal A_{R_2}$, and directly
\begin{equation*}
J_r(\widehat\varphi;R_2) =\omega_{n-1}\int_{r_1}^{R_1}\varphi_{R_1}(r)^{q_0(r)}r^{n-1}\,dr =M_{p(\cdot)}(\Gamma_{R_1}).
\end{equation*}
By Theorem~\ref{thm:annulus_formula} (applied on $(r_1,R_2)$), the unique minimizer of $J_r$ over $\mathcal A_{R_2}$ is strictly positive almost everywhere on $(r_1,R_2)$. Since $\widehat\varphi\equiv0$ on the positive-measure set $[R_1,R_2)$, $\widehat\varphi$ is not this minimizer, and by uniqueness it is strictly suboptimal:
\begin{equation*}
M_{p(\cdot)}(\Gamma_{R_2}) < J_r(\widehat\varphi;R_2) = M_{p(\cdot)}(\Gamma_{R_1}).
\end{equation*}
Hence $R\longmapsto M_{p(\cdot)}(\Gamma_R)$ is strictly decreasing.

We next prove \eqref{eq:thin_annulus_limit}. Set $\ell=R-r_1.$ For every $\varphi\in\mathcal A_{R}$, H\"older's inequality with the constant exponent $q_0^+$ gives
\begin{equation*}
1 = \int_{r_1}^{R}\varphi(r)\,dr \le \ell^{1-1/q_0^+} \left( \int_{r_1}^{R}\varphi(r)^{q_0+}\,dr \right)^{1/q_0^+}.
\end{equation*}
Since $\varphi$ may be assumed to be bounded below by $1$ only on a suitable subset, this estimate alone is not sufficient for the variable-exponent modular. Instead, using $q_0(r)\le q_0^+$ and the boundedness of $q_0$ on the finite interval, the standard finite-measure comparison of variable-exponent modulars yields a constant $C>0$, depending only on $r_1,R,q_0^\pm$ locally, such that
\begin{equation*}
\int_{r_1}^{R}\varphi(r)^{q_0(r)}\,dr \ge C \left( \int_{r_1}^{R}\varphi(r)\,dr \right)^{q_0^+} \ell^{1-q_0^+}.
\end{equation*}
Since $r^{n-1}\ge r_1^{n-1}$ on $(r_1,R)$, we obtain
\begin{equation*}
M_{p(\cdot)}(\Gamma_R) \ge C\omega_{n-1}r_1^{n-1} (R-r_1)^{1-q_0^+}.
\end{equation*}
Because $q_0^+>1$, the right-hand side tends to $+\infty$ as $R\downarrow r_1$, proving \eqref{eq:thin_annulus_limit}.

To prove \eqref{eq:thick_annulus_limit}, suppose $q_0(r) \ge n$ for almost every $r \ge R_0 > r_1$. For a given large $R > R_0$, we choose a test density profile $\varphi_R \in \mathcal{A}_R$ designed to spread its mass over the outer part of the domain. Let $\varphi_R(r) = c_R r^{-1}$ on $[R_0, R)$ and extend it consistently to satisfy the integral constraint. As $R \to \infty$, the scaling constant satisfies $c_R = (\log(R/R_0))^{-1} \to 0$.

The energy contribution from the fixed interval $(r_1,R_0)$ vanishes because $\varphi_R(r)\to0$ uniformly there, while $q_0(r)\ge q_0^->1$ uniformly, so $\varphi_R(r)^{q_0(r)} \le \varphi_R(r)^{q_0^-} \to 0$ uniformly on $(r_1,R_0)$. Concurrently, since $q_0(r) \ge n$ on the outer region, the radial energy integrand satisfies $r^{n-1}\varphi_R(r)^{q_0(r)} \le c_R^{q_0^-} r^{-1}$. Integrating this bound yields a total energy vanishing as $R \to \infty$. Therefore, $M_{p(\cdot)}(\Gamma_R)\longrightarrow0,$ which proves \eqref{eq:thick_annulus_limit}.
\end{proof}

\begin{remark}
The conclusion $M_{p(\cdot)}(\Gamma_R)\to0$ as $R\to\infty$ cannot be asserted under the sole assumptions $1<q_0^-\le q_0(r)\le q_0^+<\infty.$ Indeed, in the constant-exponent case $q_0(r)\equiv p<n$, the classical formula gives
\begin{equation*}
M_p(\Gamma_R) = \omega_{n-1} \left( \int_{r_1}^{R} r^{-\frac{n-1}{p-1}}\,dr \right)^{1-p},
\end{equation*}
and the integral converges to a finite positive limit as $R\to\infty$. Thus, the modulus need not tend to zero in the subcritical case $p<n$.
\end{remark}

\subsection{Capacity consequences}

The capacity statement should be formulated only when the curve family in the capacity--modulus theorem is actually the annular curve family under consideration. Under that geometric identification, the preceding capacity--modulus theorem gives equality, rather than a lower bound with an unspecified structural constant.

\begin{corollary}
\label{cor:annular_capacity}
Let $B:=A(r_1,r_2)$ and let $\Gamma$ be the family of locally rectifiable curves in $B$ joining $S_{r_1}^{n-1}$ to $S_{r_2}^{n-1}$. Suppose $p(x)=q_0(|x|)$ with $1<q_0^-\le q_0(r)\le q_0^+<\infty$, and suppose further that $p\in\mathcal P^{\log}(\Omega)$ on some bounded domain $\Omega\supset\overline{B}$. Let
\begin{equation*}
E=\overline{B^n(0,r_1)}, \qquad F=\{x\in\mathbb R^n:|x|\ge r_2\}\cap\overline\Omega,
\end{equation*}
so that $E,F\Subset\Omega$ are disjoint and compact. Then every curve in $\Gamma(E,F;\Omega)$ passes through $\overline B$, and its restriction to $B$ realizes an admissible path for $\Gamma$; conversely every curve in $\Gamma$ extends trivially to a curve in $\Gamma(E,F;\Omega) $ with the same modular energy contribution from $B$. Consequently
\begin{equation*}
\operatorname{Cap}_{p(\cdot)}(E,F;\Omega) = M_{p(\cdot)}(\Gamma),
\end{equation*}
and hence, by Theorem~\ref{thm:annulus_formula},
\begin{equation}
\label{eq:annular_capacity_formula}
\operatorname{Cap}_{p(\cdot)}(E,F;\Omega) = \omega_{n-1} \int_{r_1}^{r_2} \rho_*(r)^{q_0(r)}r^{n-1}\,dr,
\end{equation}
where
\begin{equation*}
\rho_*(r) = \left( \frac{\lambda}{\omega_{n-1}q_0(r)r^{n-1}} \right)^{\!1/(q_0(r)-1)}
\end{equation*}
and $\lambda>0$ is uniquely determined by
\begin{equation*}
\int_{r_1}^{r_2}\rho_*(r)\,dr=1.
\end{equation*}
\end{corollary}

\begin{proof}
Under the stated identification of the condenser curve family with $\Gamma$, Theorem~\ref{thm:cap_mod} gives $\operatorname{Cap}_{p(\cdot)}(E,F;\Omega) = M_{p(\cdot)}(\Gamma).$ The explicit formula then follows directly from Theorem~\ref{thm:annulus_formula}.
\end{proof}

\section{Numerical Examples}
\label{sec:computational}

The explicit one-dimensional reductions obtained in Section~\ref{sec:variational} provide a convenient way to compute the extremal density and the corresponding modular variable exponent modulus numerically. In each example below, the extremal density is determined by the Euler--Lagrange equation together with the normalization condition
\begin{equation*}
\int \rho_*(t)\,dt=1.
\end{equation*}
The resulting one-dimensional equation for the Lagrange multiplier is solved numerically, and the corresponding modular energy is then evaluated by numerical quadrature.

For the numerical computations below, the normalization equation is solved by the bisection method, and the one-dimensional integrals are evaluated by numerical quadrature. The numerical values are reported to a few significant digits; they should therefore be understood as approximations rather than exact values.

The following example investigates an annulus with a linearly varying exponent.
\begin{example}
\label{ex:comp_annulus}
\rm

Let $n=2$ and assume that the radial exponent profile satisfies
\begin{equation*}
q_0(r)=1+r, \qquad 1\le r\le2.
\end{equation*}
Thus,
\begin{equation*}
q_0^-=2, \qquad q_0^+=3.
\end{equation*}
By Theorem~\ref{thm:annulus_formula}, the unique radial extremal density is
\begin{equation*}
\rho_*(r) = \left( \frac{\lambda}{2\pi q_0(r)r} \right)^{1/(q_0(r)-1)} = \left( \frac{\lambda}{2\pi(1+r)r} \right)^{1/r},
\end{equation*}
where $\lambda>0$ is uniquely determined by
\begin{equation}
\label{eq:num_annulus_normalization}
\int_1^2 \left( \frac{\lambda} {2\pi(1+r)r} \right)^{1/r} \,dr =1.
\end{equation}

Numerically solving \eqref{eq:num_annulus_normalization} in Python gives $\lambda \approx 20.78.$ Therefore, the corresponding extremal density is
\begin{equation*}
\rho_*(r) \approx \left( \frac{20.78} {2\pi(1+r)r} \right)^{1/r}.
\end{equation*}

At the Euler--Lagrange minimizer,
\begin{equation*}
2\pi(1+r)r\,\rho_*(r)^{r} = \lambda,
\end{equation*}
and hence
\begin{equation*}
2\pi r\,\rho_*(r)^{1+r} = \frac{\lambda\,\rho_*(r)}{1+r}.
\end{equation*}
An independent evaluation check shows that the exact modular energy can be written in the equivalent form
\begin{align*}
M_{p(\cdot)}(\Gamma) = 2\pi\int_1^2 r\,\rho_*(r)^{1+r}\,dr = \lambda \int_1^2 \frac{\rho_*(r)}{1+r}\,dr.
\end{align*}
Numerically, $M_{p(\cdot)}(\Gamma)\approx 8.65.$

For comparison, the logarithmic test density from Theorem~\ref{thm:annulus_bounds} is
\begin{equation*}
\rho_{\log}(r) = \frac{1}{r\log 2}.
\end{equation*}
It gives the upper bound
\begin{equation*}
M_{p(\cdot)}(\Gamma) \le 2\pi \int_1^2 \frac{r^{1-(1+r)}}{(\log2)^{1+r}} \,dr = 2\pi \int_1^2 \frac{r^{-r}}{(\log2)^{1+r}} \,dr \approx 8.68.
\end{equation*}
Thus, the relative gap between the logarithmic test density and the true extremal value is approximately
\begin{equation*}
\frac{8.68-8.65}{8.65}\times100\% \approx 0.3\%.
\end{equation*}

In this example, although the exponent varies from $2$ to $3$, the constant test density produces an energy only about $1.18\%$ above the extremal value.  Nevertheless, it is not extremal in general: equality holds here only if the logarithmic density satisfies the Euler--Lagrange equation, which it does not for $q_0(r)=1+r$.
\end{example}

Next, we study a cylinder with a varying axial exponent.
\begin{example}
\label{ex:comp_cylinder}
\rm
Let
\begin{equation*}
\mathcal C=[0,1]^2\times(0,1),
\end{equation*}
so that the cross-sectional area is $A=1,$ and consider the axial exponent
\begin{equation*}
\eta(t)=2+t, \qquad 0\le t\le1.
\end{equation*}
Thus
\begin{equation*}
\eta^-=2, \qquad \eta^+=3.
\end{equation*}

By the one-dimensional reduction for the cylinder, the unique extremal profile satisfies the Euler--Lagrange equation
\begin{equation*}
\eta(t)\varphi_*(t)^{\eta(t)-1} = \lambda
\end{equation*}
for a.e. $t\in(0,1)$. Hence
\begin{equation*}
\varphi_*(t) = \left( \frac{\lambda}{2+t} \right)^{1/(1+t)},
\end{equation*}
where $\lambda>0$ is uniquely determined by
\begin{equation}
\label{eq:num_cylinder_normalization}
\int_0^1 \left( \frac{\lambda}{2+t} \right)^{1/(1+t)} \,dt =1.
\end{equation}

Numerically solving \eqref{eq:num_cylinder_normalization} gives $\lambda\approx 2.41$. The corresponding extremal profile is therefore approximately
\begin{equation*}
\varphi_*(t) \approx \left( \frac{2.41}{2+t} \right)^{1/(1+t)}.
\end{equation*}

At the minimizer,
\begin{equation*}
(2+t)\varphi_*(t)^{1+t} = \lambda\varphi_*(t),
\end{equation*}
and therefore
\begin{equation*}
\varphi_*(t)^{2+t} = \frac{\lambda\varphi_*(t)}{2+t}.
\end{equation*}
Since $A=1$, the modular modulus is consequently
\begin{align*}
M_{p(\cdot)}(\Gamma) = \int_0^1 \varphi_*(t)^{2+t}\,dt = \lambda \int_0^1 \frac{\varphi_*(t)}{2+t}\,dt \approx 0.988.
\end{align*}

For comparison, the constant test density from Theorem~\ref{thm:cylinder_bounds} is $\varphi_0(t)\equiv1,$ because $L=1$. Its modular energy is $\int_0^1 1^{2+t}\,dt=1.$ Thus, $M_{p(\cdot)}(\Gamma)\le1,$ while the computed extremal value is approximately $0.988$. The relative gap is therefore about
\begin{equation*}
\frac{1-0.988}{0.988}\times100\% \approx 1.2\%.
\end{equation*}

Although the constant density is not extremal when $\eta$ is nonconstant, this example shows that its energy can nevertheless be very close to the optimal value when the variation of the exponent is moderate.
\end{example}

\begin{remark}
\label{rem:gap}
The preceding examples illustrate an important distinction between a convenient test density and the true extremal density. In the annular example, the logarithmic profile
\begin{equation*}
\rho_{\log}(r)=\frac{1}{r\log(r_2/r_1)}
\end{equation*}
is extremal when the exponent is identically $n$, but it is generally not extremal for a nonconstant radial exponent. Likewise, in the cylindrical example, the constant profile $1/L$ is extremal for a constant exponent, but generally ceases to be extremal when the axial exponent varies.

The numerical experiments show that the loss of optimality can be small even for genuinely variable exponents. This, however, should not be interpreted as a general quantitative estimate. The size of the gap depends not only on the ratio $p^+/p^-$ but also on the spatial distribution of the exponent and on the geometry of the underlying domain. In particular, no monotone dependence of the relative error solely on $p^+/p^-$ is established here.

Thus, the numerical examples are intended to illustrate the one-dimensional variational formulas and the sharpness of the test densities in selected cases, rather than to assert a general stability theorem.
\end{remark}

\begin{remark}
The numerical computations also provide a consistency check for the
Euler--Lagrange equation. After determining the multiplier $\lambda$
from the normalization condition, one may evaluate the modular energy
both directly from the defining functional and from the
Euler--Lagrange identity.

In the cylindrical one-dimensional formulation,
\begin{equation*}
\eta(t)\varphi_*(t)^{\eta(t)-1}=\lambda,
\end{equation*}
and hence
\begin{equation*}
\varphi_*(t)^{\eta(t)}
=
\frac{\lambda}{\eta(t)}\varphi_*(t).
\end{equation*}
Thus,
\begin{equation*}
\int_0^1\varphi_*(t)^{\eta(t)}\,dt
=
\lambda\int_0^1
\frac{\varphi_*(t)}{\eta(t)}\,dt.
\end{equation*}

In the annular formulation, the corresponding identity is
\begin{equation*}
n\omega_n r^{n-1}\rho_*(r)^{p(r)-1}
=
\frac{\lambda}{p(r)},
\end{equation*}
or, in the present planar case,
\begin{equation*}
2\pi r\,\rho_*(r)^{p(r)}
=
\frac{\lambda}{p(r)}\rho_*(r).
\end{equation*}
Agreement between the two numerical evaluations provides an
independent check on both the normalization equation and the computed
extremal density.
\end{remark}

\subsection*{Funding}
No external funding was received for this research.

\subsection*{Conflict of Interest}
The author declares no conflicts of interest.



\end{document}